\documentclass[12pt,leqno,twoside]{amsart}
\usepackage{amssymb,amsmath,amsthm,soul,color}
\usepackage{t1enc}
\usepackage[cp1250]{inputenc}
\usepackage{a4,indentfirst,latexsym}
\usepackage{graphics}
\usepackage{mathrsfs}
\usepackage{cite,enumitem,graphicx}
\usepackage[colorlinks=true,urlcolor=blue,
citecolor=red,linkcolor=blue,linktocpage,pdfpagelabels,
bookmarksnumbered,bookmarksopen]{hyperref}
\usepackage[english]{babel}
\usepackage[left=2.61cm,right=2.61cm,top=2.72cm,bottom=2.72cm]{geometry}
\usepackage[metapost]{mfpic}
\usepackage[normalem]{ulem}

\newtheorem{Th}{Theorem}[section]
\newtheorem{Prop}[Th]{Proposition}
\newtheorem{Lem}[Th]{Lemma}
\newtheorem{Cor}[Th]{Corollary}

\newtheorem{Rem}[Th]{Remark}

\newenvironment{altproof}[1]
{\noindent
{\em Proof of {#1}}.}
{\nopagebreak\mbox{}\hfill $\Box$\par\addvspace{0.5cm}}

   \newcommand{\vp}{\varphi}
   
   \newcommand{\eps}{\varepsilon}

   \def\div{\mathop{\mathrm{div}\,}}

   \def\supp{\mathrm{supp}}

   \def\N{\mathbb{N}}
   \def\R{\mathbb{R}}

   \def\curl{\mathrm{curl}}
   \def\dim{\mathrm{dim}}

   \def\cl{\mathrm{cl\,}}

   \def\P{\mathcal P} 
   
   \def\V{\mathcal{V}}

   \def\W{\mathcal{W}}
   \def\D{\mathcal{D}}

\newcommand{\cC}{{\mathcal C}}

\newcommand{\cE}{{\mathcal E}}

\newcommand{\cN}{{\mathcal N}}
\newcommand{\cO}{{\mathcal O}}
\newcommand{\cP}{{\mathcal P}}

\newcommand{\cT}{{\mathcal T}}

\newcommand{\cV}{{\mathcal V}}
\newcommand{\cW}{{\mathcal W}}

\renewcommand{\dim}{{\rm dim}\,}

\newcommand{\al}{\alpha}
\newcommand{\be}{\beta}
\newcommand{\ga}{\gamma}

\newcommand{\la}{\lambda}

\newcommand{\Om}{\Omega}

\def\curlop{\nabla\times}
\newcommand{\weakto}{\rightharpoonup}
\newcommand{\pa}{\partial}

\newcommand{\tX}{\widetilde{X}}

\newcommand{\tu}{\widetilde{u}}
\newcommand{\tv}{\widetilde{v}}
\newcommand{\tcV}{\widetilde{\cV}}

\newcommand{\cTto}{\stackrel{\cT}{\longrightarrow}}

\numberwithin{equation}{section}

\begin{document}
\title{The Brezis-Nirenberg problem for the curl-curl operator}
\author[J. Mederski]{Jaros\l aw Mederski}
\address[J. Mederski]{\newline\indent 
Institute of Mathematics,
\newline\indent
Polish Academy of Sciences,
\newline\indent 
ul. \'Sniadeckich 8, 00-956
Warszawa, Poland
\newline\indent
and
\newline\indent
Faculty of Mathematics and Computer Science,
\newline\indent 
Nicolaus Copernicus University,
\newline\indent
ul. Chopina 12/18, 87-100 Toru\'n, Poland}
\email{\href{mailto:jmederski@mat.umk.pl}{jmederski@impan.pl}}
\maketitle

\pagestyle{myheadings} \markboth{\underline{J. Mederski}}{
\underline{The Brezis-Nirenberg problem for the curl-curl operator}}

\begin{abstract}
We look for solutions $E:\Omega\to\mathbb{R}^3$ of the problem
\[
\left\{
\begin{aligned}
&\nabla\times(\nabla\times E) +\lambda E = |E|^{p-2}E &&\quad \text{in }\Omega\\
&\nu\times E = 0 &&\quad \text{on }\partial\Omega
\end{aligned}
\right.
\]
on a bounded Lipschitz domain $\Omega\subset\mathbb{R}^3$, where  $\nabla\times$ denotes the curl operator in $\mathbb{R}^3$. The equation describes the propagation of the time-harmonic electric field $\Re\{E(x)e^{i\omega t}\}$ in a nonlinear isotropic material $\Omega$ with $\lambda=-\mu \varepsilon \omega^2\leq 0$, where $\mu$ and $\varepsilon$ stand for
the permeability and the linear part of the permittivity of the material. The nonlinear term $|E|^{p-2}E$ with $p>2$ is responsible for the nonlinear polarisation of $\Omega$ and
the boundary conditions are those for $\Omega$ surrounded by a perfect conductor. The problem has a variational structure and we deal with the critical value $p$, for instance, in convex domains $\Omega$ or in domains with $\mathcal{C}^{1,1}$ boundary, $p=6=2^*$ is the Sobolev critical exponent and we get the quintic nonlinearity in the equation. We show that there exist a cylindrically symmetric ground state solution and a finite number of cylindrically symmetric bound states depending on $\lambda\leq 0$. We develop a new critical point theory which allows to solve the problem, and which enables us to treat more general anisotropic media as well as other variational problems.
\end{abstract}

{\bf MSC 2010:} Primary: 35Q60; Secondary: 35J20, 58E05, 35B33, 78A25

{\bf Key words:} time-harmonic Maxwell equations, perfect conductor, ground state, variational methods, strongly indefinite functional, Nehari-Pankov manifold, Brezis-Nirenberg problem, critical exponent.

\section{Introduction}\label{sec:intor}

The following equation
$$
\curlop\left(\mu^{-1}\curlop \cE\right)+\eps\partial_t^2 \cE
 = -\partial_t^2 \cP_{NL}.
$$
describes the propagation of the electric field $\cE$ in a nonlinear bounded medium $\Om$ with the permeability $\mu$,  the linear part of the permittivity $\eps$ and the nonlinear polarisation $\P_{NL}$; see Saleh and Teich \cite{FundPhotonics}. In the time-harmonic case the fields $\cE$ and $\cP_{NL}$ are of the form
$\cE(x,t) = E(x)e^{i\omega t}$, $\cP_{NL}(x,t) = P_{NL}(x)e^{i\omega t}$,
which leads to the time-harmonic Maxwell equation
\begin{equation}\label{eq:timeMaxwelleq}
\curlop\left(\mu^{-1}\curlop E\right)-\omega^2\eps E = \omega^2 P_{NL}.
\end{equation}
Since $P_{NL}$ depends on $E$, and assuming $\eps,\mu>0$ to be constant, we finally obtain an equation of the form
\begin{equation}
\curlop(\curlop E) + \la E = f(x,E) \qquad\textnormal{in } \Om,
\end{equation}
where $\la = -\mu\omega^2\eps\leq 0$ and $f(x,E)=\mu\omega^2 P_{NL}$. In particular, we concentrate on the following problem
\begin{equation}\label{eq:main}
\curlop(\curlop E) + \la E = |E|^{p-2}E \qquad\textnormal{in } \Om
\end{equation}
with $p>2$ and together with the boundary condition
\begin{equation}\label{eq:bc}
\nu\times E = 0\qquad\text{on }\pa\Om
\end{equation}
where $\nu:\pa\Om\to\R^3$ is the exterior normal. This boundary condition holds when $\Om$ is surrounded by a perfect conductor; see for instance \cite{Monk,Doerfler,BartschMederski1} and references therein.\\
\indent The linear time-harmonic Maxwell equations, i.e. when $P_{NL}=0$, have been extensively studied, e.g. \cite{BS,Monk,Leis68,Mitrea04,Picard01,Doerfler,Vico}, however the nonlinear case is still investigated only tangentially in the mathematical literature. Firstly, we would like to mention that if $\Om=\R^3$ then cylindrically symmetric transverse electric and transverse magnetic solutions have been considered in  a series of papers by Stuart and Zhou \cite{StuartZhou96,Stuart91,StuartZhou05,StuartZhou10,StuartZhou03,StuartZhou01,Stuart04}  for asymptotically linear $P_{NL}$  and by McLeod,  Stuart and Troy \cite{McLeodStuartTroy} for a cubic nonlinear polarization. 
The search for these solutions reduces to a one-dimensional varia\-tional problem or an ODE, which simplifies the problem considerably. The ODE methods, however, seem to be difficult to apply to our problem \eqref{eq:main}-\eqref{eq:bc}, since we easily show that nontrivial radial solutions do not exist. Indeed, if $E\in L_{loc}^{p-1}(\Om,\R^3)$ is a distributional solution of \eqref{eq:main} such that $E(x)=M^{T}E(Mx)$ for a.e. $x\in\Om$ and all $M\in\cO(3)$, then $\curlop E=0$ and $\lambda\geq 0$  similarly as in \cite{Bartsch:2014}[Theorem 1], hence  $\lambda=0$ and $E=0$.\\
\indent Recall that in \cite{BartschMederski1}, Bartsch and the author have dealt with general subcritical nonlinearities of the form $f(x,E)=\pa_E F(x,E)$ on a convex domain $\Om$ or on a simply-connected domain with the connected $\cC^{1,1}$ boundary, having $F(x,E)=\frac1p|E|^p$ with $2<p<6=2^*$ as a model in mind; see also the survey \cite{BartschMederskiSurvey} on the curl-curl problems in the subcritical case. The first goal of this paper is to find weak solutions to \eqref{eq:main} for the critically growing nonlineari\-ty with $p=6$ on such media $\Om$, which have not been considered in the mathematical literature so far. In the physical context, the critical nonlinearity represents the focusing quintic effect of the material and the nonlinear effect usually takes the form $ f(x,E) =  \chi^{(5)}|E|^4E-\chi^{(3)}|E|^2E$, where  $\chi^{(3)}$, $ \chi^{(5)}$ are corresponding susceptibility parameters. In this work we are able to deal with  the case, where $\chi^{(5)}>0$  and $\chi^{(3)}=0$.
Moreover we work on general Lipschitz domains with the following nonlinearity
\begin{equation}\label{eqCriticalrNonlinearity2}
F(x,E)=\frac{1}{p}|E|^{p},\quad  p=\frac{6}{3-2s},
\end{equation}  
where
$$X_N(\Om):=\left\{E\in H_0(\curl;\Om): \div(E)\in L^2(\Om,\R^3)\right\}$$
embeds continuously into $H^s(\Om,\R^3)$ for some $s\in [1/2,1]$ and $H_0(\curl;\Om)$ is the natural space for $\curlop(\curlop\cdot)$ operator with \eqref{eq:bc}; see \cite{Amrouche} and Section \ref{sec:varsetting} for details. Here $p$ is such that $H^s(\Om,\R^3)$ embeds continuously into $L^p(\Om,\R^3)$ but not necessarily compactly. For instance, for convex $\Om$ or for domains with $\cC^{1,1}$ boundary one has $s=1$ and $p=6$. Therefore equation \eqref{eq:main} is a three-dimensional variant of the well-known Brezis-Nirenberg problem \cite{BrezisNirenberg}
\begin{equation}\label{eq:BNproblem}
-\Delta u+ \lambda u = |u|^{2^*-2}u\quad\hbox{ for }u\in H_0^1(\Om),
\end{equation}
where $\Om$ is a bounded domain in $\R^N$, $N\geq 3$ and $2^*=\frac{2N}{N-2}$ is the critical Sobolev constant.\\
\indent On a suitable subspace $X\subset H_0(\curl;\Om)$ defined in Section \ref{sec:results}, weak solutions of \eqref{eq:main} correspond to critical points 
of the associated energy functional
\begin{equation}\label{eq:defOfJ}
J_\lambda(E)=\frac12\int_\Om |\curlop E|^2\, dx + \frac{\lambda}{2}\int_{\Om} |E|^2\, dx-\frac{1}{p}\int_\Om |E|^p\, dx,
\end{equation}  
$J_\lambda$ is unbounded from above and from below, even on subspaces of finite codimension and its critical points have infinite Morse index. Therefore the problem has the strongly indefinite nature. This is due to the fact that $\curlop(\curlop\cdot)$ has an infinite dimensional kernel, since $\curlop(\nabla\vp)=0$ for $\vp\in \cC_0^{\infty}(\Om)$. Although for $\lambda\leq 0$, $J_\lambda$  has a linking geometry in the spirit of Benci and Rabinowitz \cite{BenciRabinowitz,Rabinowitz:1986} we cannot apply these results to get a Palais-Smale sequence, since $J'$ is not (sequentially) weak-to-weak$^*$ continuous, i.e. the weak convergence $E_n\weakto E$ in $X$ does not imply that $J'_\lambda(E_n)\overset{\ast}{\weakto} J'_\lambda(E)$ in $X^*$; see also \cite{MederskiNLSS} for the recent linking results and the references therein, where this regularity is required. Moreover, even if we find somehow a bounded Palais-Smale sequence $E_n\weakto E$ we do not know whether $E$ is a critical point of  $J_\lambda$. This is caused, again, by the lack of  weak-to-weak$^*$ continuity of $J_\lambda'$.\\
\indent In order to find a Palais-Smale sequence we minimize $J_\lambda$ on a natural constraint $\cN_\lambda$, called the Nehari-Pankov manifold, which is contained in the usual Nehari manifold $\{E\neq 0|\; J'_\lambda(E)(E)=0\}$ inspired by works \cite{Pankov,SzulkinWeth}. As in \cite{BartschMederski1} we show that $\cN_\lambda$ is homeomorphic with the unit sphere in a subspace of $X_N$ consisting of divergence-free vector fields. This allows to find a minimizing sequence on the sphere, which is of $\cC^1$ class, and hence on the Nehari-Pankov manifold using the homeomorphism. However in \cite{BartschMederski1} for $p<6$ and in \cite{BartschMederski2} for more general materials and $p<\frac{6}{3-2s}$ we have been in a position to find a limit point of the sequence being a nontrivial critical point, since $X_N$ is compactly embedded into $L^p(\Om,\R^3)$. 
The methods of \cite{BartschMederski1,BartschMederski2} are no longer applicable in the critical case \eqref{eqCriticalrNonlinearity2}.\\
\indent Inspired by Brezis and Nirenberg \cite{BrezisNirenberg}, it would seem that a natural approach to solve \eqref{eq:main} is to find a Palais-Smale sequence below the energy level of a certain least energy solution $E:\R^3\to\R^3$ of the following limiting problem
\begin{equation}\label{limitingproblemR3}
\curlop(\curlop E)= |E|^{p-2}E \quad \text{in }\R^3,
\end{equation}
and show that the Palais-Smale condition holds in this case. Recall that the corresponding comparison of minimization problems concerning \eqref{eq:BNproblem} in \cite{BrezisNirenberg} is strongly based on the shape of all solutions of the limiting problem $-\Delta u=|u|^{2^*-2}u$ with $u\in \D^{1,2}(\R^N)$, called the Aubin-Talenti instantons \cite{Aub,Tal}, which are radial and given by the explicit formula. Moreover the lowest dimensional case $N=3$ for \eqref{eq:BNproblem} is the most challenging one to get the comparison of the minimization problems and to get the compactness of Palais-Smale sequences. Hence, in order to try to adopt this approach to our problem one should have the precise information about the shape of solutions of \eqref{limitingproblemR3}. However, in view of \cite{MederskiENZ}[Corollary 2.5] one sees that \eqref{limitingproblemR3} admits no classical solutions for
$p\neq 6$ and $p\geq 2$. Moreover, similarly as in\cite{Bartsch:2014}[Theorem 1] we show that for any $p\geq 2$ there are no radial weak solutions, so that the usual ODE methods are not applicable to find solutions to \eqref{limitingproblemR3} and their shapes. Note that a semilinear equation involving $\curlop(\curlop\cdot)$ in $\R^3$ has been also considered in 
\cite{BenFor,DAprileSiciliano,BenForAzzAprile} with different nonlinearities of subcritical growth for large fields and in a different physical context, as well as in \cite{Bartsch:2014,HirschReichel} with a subcritical nonlinearity and a similar physical motivation. Taking into account the methods of these works, instead of radial symmetry one could try to find cylindrically symmetric solutions to \eqref{limitingproblemR3} of the form
\begin{equation}\label{eq:symR^3}
E(x)=\frac{u(r,x_3)}{r}\begin{pmatrix}-x_2\\x_1\\0\end{pmatrix},\qquad x=(x_1,x_2,x_3)\in\R^3\hbox{ and } r=\sqrt{x_1^2+x_2^2},
\end{equation}
for some $u:(0,+\infty)\times\R\to\R$. By the direct computations we observe  that a field $E$ of the form \eqref{eq:symR^3} solves \eqref{limitingproblemR3} if and only if $\phi(x)=u(r,x_3)$ solves
\begin{equation}\label{eq:scalarGeneral}
-\Delta \phi + \frac{\phi}{r^2}= |\phi|^{p-2}\phi\quad\hbox{in }\R^3.
\end{equation}
Surprisingly, the last equation with $p=6$ has been also obtained by Esteban and Lions in  \cite{EstebanLions}[Theorem 3.9] as the limiting problem for a nonlinear Schr\"odinger equation of critical growth with an external magnetic field. To the best of our knowledge, no precise information about the shape of solutions of \eqref{eq:scalarGeneral} with $p=6$ has been provided, which would allow to adopt the methods of \cite{BrezisNirenberg}.\\
\indent Since the problem of the existence and information about the shape of solution to \eqref{limitingproblemR3}  is still open, we introduce a different approach than in \cite{BrezisNirenberg}, or in the related works; see for instance \cite{BartoloNA83,Comte,CeramiAIHP1984,CeramiJFA1986,DevilSolimini,Solimini} and references therein. Namely, we analyse the monotonicity of ground state levels 
$$c_\lambda:=\inf_{\cN_\lambda}J_\lambda\quad\hbox{ for }\lambda\in (-\infty,0],$$
and, roughly speaking, we show that if $c_\lambda$ is strictly increasing in some open interval $I\subset  (-\infty,0]$, then we find a Palais-Smale sequence contained in $\cN_\lambda$ at level $c_\lambda$ with a nontrivial weak limit point for all $\lambda\in I$. In the critical case, however, we do not know whether the limit point is a critical point due to the lack of the weak-to-weak$^*$ continuity of $J_\lambda'$, even on $\cN_\lambda$. In order to find critical points of $J_\lambda$ and solve \eqref{eq:main}, we restrict our approach to the subspace $X^{cyl}$ of cylindrically symmetric vector fields $E\in X^{cyl}\subset \V$ of the form
\begin{equation}\label{eq:sym}
E(x)=u(r,x_3)\begin{pmatrix}-x_2\\x_1\\0\end{pmatrix},\qquad x=(x_1,x_2,x_3)\in\Om\hbox{ and } r=\sqrt{x_1^2+x_2^2},
\end{equation}
for some $u:(0,+\infty)\times\R\to\R$, where $\Om$ is assumed to be invariant under the orthogonal group action $G = \cO(2)\times{1}\subset \cO(3)$. Since $J_\lambda'$ is weak-to-weak$^*$ continuous in $\V$, then the analysis of the monotonicity of symmetric ground state levels 
$$c_\lambda^{cyl}:=\inf_{\cN_\lambda^{cyl}}J_\lambda\quad\hbox{ for }\lambda\in (-\infty,0],$$
where $\cN_\lambda^{cyl}$ is the Nehari-Pankov manifold for $J_\lambda|_{X^{cyl}}$,
shows that $c_\lambda^{cyl}$ is attained in some  open intervals $I\subset  (-\infty,0]$. In order to find such intervals we compare $c_\lambda^{cyl}$ with $c_0^{cyl}$, the new limiting problem \eqref{eq:main} at level $\lambda=0$ still on the same bounded domain. Moreover, if $-\lambda$ is in a right neighbourhood of an eigenvalue of $\curlop(\curlop(\cdot))$ in $X^{cyl}$, then the number of solutions is bounded from below by the multiplicity of the eigenvalue. The precise statement of our results concerning \eqref{eq:main} is presented in the next Section \ref{sec:results}.\\
\indent The remaining part of the paper is organised as follows. In Section \ref{sec:Nehari} we build an abstract critical point theory which enables us to analyse the monotonicity of ground state levels of general strongly indefinite functionals on the Nehari-Pankov manifold, and which provides the information about the existence and the multiplicity of critical points. Section \ref{sec:varsetting} is devoted to proof of the results from Section \ref{sec:results}.  Our critical point theory  can be applied to other variational problems and 
in the last Section \ref{sec:AM} we study the time-harmonic Maxwell equation in more general anisotropic media.

\section{Statement of results}\label{sec:results}

Throughout the paper we assume that $\Om\subset\R^3$ is a bounded Lipschitz domain. Recall that the proper space for the curl-curl eigenvalue problem is
$$
H(\curl;\Om) := \{E\in L^2(\Om,\R^3): \curlop E \in L^2(\Om,\R^3)\},
$$
which is a Hilbert space when provided with the graph norm
$$
\|E\|_{H(\curl;\Om)} := \left(|E|^2_2+|\curlop E|^2_2\right)^{1/2}.
$$
Here and in the sequel $|\cdot|_q$ denotes the $L^q$-norm for $q\geq 1$, $q=\infty$. The curl of $E$, $\curlop E$, has to be understood in the distributional sense. The closure of $\cC^{\infty}_0(\Om,\R^3)$ in $H(\curl;\Om)$ is denoted by $H_0(\curl;\Om)$.\\
\indent We also need the subspace
\[
\cV := \left\{v\in H_0(\curl;\Om): \int_\Om\langle v,\varphi\rangle\,dx=0
        \text{ for every $\varphi\in \cC^\infty_0(\Om,\R^3)$ with $\curlop\varphi=0$} \right\}.
\]
and observe that
\[
\begin{aligned}
\cV
&\subset\left\{E\in H_0(\curl;\Om): \div(E)\in L^2(\Om,\R^3)\right\} =: X_N(\Om).
\end{aligned}
\]
In view of \cite{Amrouche,Costabel}, we know that
$X_N(\Om)$ embeds continuously into $H^{s}(\Om,\R^3)$ for some $s\in [1/2,1]$, but if, in addition $\Om$ is convex or has $\cC^{1,1}$-boundary, then $X_N(\Om)$ embeds into $H^{1}(\Om,\R^3)$. Thus we work with the following general assumption.
\begin{itemize}
\item[$(V)$] $\cV$ is continuously embedded into $L^p(\Om,\R^3)$ 
for some $2<p\leq 6$.
\end{itemize}
Below we show that the embedding $\cV\subset L^p(\Om,\R^3)$ is not  compact for $p=2^*=6.$
\begin{Rem} \label{RemMainRes}
Suppose that $\Om$ is star-shaped, i.e. there is $x_0\in \Om$ such that $x_0+t(x-x_0)\in\Om$ for any $t\in [0,1]$ and $x\in\Om$. 
Let $E\in H_0(\curl;\Om)$ and take $\vp_n\in\cC_0^{\infty}(\Om,\R^3)$ such that $E=\lim_{n\to\infty}\vp_n$ in $H_0(\curl;\Om)$.
For any $\eps\in (0,1]$ and $x\in\Om$ we set
\[
\vp_{n,\eps}(x)=\left\{
\begin{aligned}
&\eps^{-1/2}\vp_n((x-x_0)/\eps) &&\quad 
\text{if } (x-x_0)/\eps\in\Om\\
&0 &&\quad \text{if } (x-x_0)/\eps\notin\Om.
\end{aligned}
\right.
\]
We easily check that 
$\supp(\vp_{n,\eps})\subset \Om_\eps\subset \Om$ and $\vp_{n,\eps}\in\cC_0^{\infty}(\Om,\R^3)$ for sufficiently small $\eps>0$,  where
$\Om_\eps=\{x_0+\eps x: x\in\Om\}$.
 Observe that for $n,m\geq 1$ one has
\begin{eqnarray*}
\|\vp_{n,\eps}-\vp_{m,\eps}\|^2
&=& |\curlop (\vp_{n}-\vp_{m})|_2^2 +\eps^2 |\vp_{n}-\vp_{m}|_2^2.
\end{eqnarray*}
Hence  $(\vp_{n,\eps})$ is a Cauchy sequence and
we find the limit 
\begin{equation}\label{eq:defOfE_eps}
E_\eps=\lim_{n\to\infty}\vp_{n,\eps}\hbox{ in }H_0(\curl;\Om).
\end{equation}
 Observe that if $E\in\cV$, then for any $\vp\in\cC_0^{\infty}(\Om,\R^3)$ such that $\curlop \vp =0$ one has
\begin{eqnarray*}
\int_\Om\langle E_\eps,\varphi\rangle\,dx&\leq&
|E_\eps-\vp_{n,\eps}|_2|\vp|_2+\int_{\Om_\eps}\langle \eps^{-1/2}\vp_{n}((x-x_0)/\eps),\vp(x)\rangle\,dx\\
&\leq&
|E_\eps-\vp_{n,\eps}|_2|\vp|_2+\eps^{3-1/2}\int_\Om\langle \vp_{n}(y),\vp(x_0+\eps y)\rangle\,dy\\
&\to& \eps^{5/2}\int_\Om\langle E(y),\vp(x_0+\eps y)\rangle\,dy = 0
\end{eqnarray*}
as $n\to\infty$, since $\curlop\vp(x_0+\eps(\cdot))=0$. Thus $E_\eps\in\V$. Now let $E\in \cV\setminus\{0\}$, $\cV$ embeds continuously in $L^6(\Om,\R^3)$ and $\eps\to 0$. Then for any $\vp\in\cC_0^{\infty}(\Om,\R^3)$ we show that
$$\int_{\Om}\langle \curlop E_\eps, \vp \rangle\, dx\to 0,$$
hence $E_\eps\weakto 0$ in $\V$.
On the other hand
$$|E_\eps|_6^6=\lim_{n\to\infty}\int_{\Om_\eps} \eps^{-3}|\vp_{n}((x-x_0)/\eps)|^6\, dx=\lim_{n\to\infty}|\vp_n|_6^6=|E|_6^6,$$
and the embedding $\cV\subset L^6(\Om,\R^3)$ cannot be compact.
\end{Rem}
In order to state our results we introduce the space
$$
W^p(\curl;\Om) := \{E\in L^p(\Om,\R^3): \curlop E\in L^2(\Om,\R^3)\}\subset H(\curl;\Om)
$$
which is a Banach space if provided with the norm
$$
\|E\|_{W^p(\curl;\Om)} := \left(|E|^2_p+|\curlop E|^2_2\right)^{1/2}.
$$
We shall look for solutions of \eqref{eq:main} in the closure
$$X:=W^p_0(\curl;\Om)\subset H_0(\curl;\Om)$$
of $\cC^{\infty}_0(\Om,\R^3)$ in $W^p(\curl;\Om)$. Observe that $\cV$ is a closed linear subspace of $W^p_0(\curl;\Om)$ as a consequence of $(V)$. Moreover, since for every $\varphi\in\cC^\infty_0(\Om,\R^3)$ the linear map
\[
E \mapsto \int_\Om \langle E,\curlop\varphi\rangle dx
\]
is continuous on $W^p_0(\curl;\Om)\subset H(\curl;\Om)$, the space
\[
\begin{aligned}
\cW
 &:= \left\{w\in W^p_0(\curl;\Om):\int_\Om\langle w,\curlop\varphi\rangle = 0
     \text{ for all }\varphi\in\cC^\infty_0(\Om,\R^3)\right\}\\
 &= \{w\in W^p_0(\curl;\Om): \curlop w=0\}
\end{aligned}
\]
is a closed complement of $\cV$ in $W^p_0(\curl;\Om)$ (cf. \cite{KirschHettlich}[Theorem 4.21 c)]), hence there is a Helmholtz type decomposition 
$$X=W^p_0(\curl;\Om) = \cV\oplus\cW.$$
\indent We will show that the spectrum of the curl-curl operator in $H_0(\curl;\Om)$ consists of the eigenvalue $\lambda_0=0$ with infinite multiplicity and of a sequence of eigenvalues 
$$0<\la_1\le \la_2\le\dots\le\la_k\to\infty$$ with the corresponding finite multiplicities $m(\lambda_k)\in\N$. For $\lambda\leq 0$ we find two closed and orthogonal subspaces $\cV^+$ and $\tcV$ of $\V$ such that the quadratic form $Q:\cV\to\R$
given by
\begin{equation}\label{eq:DefQ}
Q(v)=\int_\Om |\curlop v|^2+\lambda |v|^2\, dx
\end{equation}
is positive on $\cV^+$ and semi-negative on $\tcV$, where $\dim\tcV<\infty$. Now we can define the so-called {\em Nehari-Pankov manifold}
\begin{equation}\label{eq:DefOfN}
\cN_\lambda=\{E\in X\setminus(\tcV\oplus \W):\; J_\lambda'(E)|_{\R E\oplus\tcV\oplus \W}=0\},
\end{equation}
being a closed subset of the usual Nehari manifold 
$$\{E\in X\setminus\{0\}:\; J_\lambda'(E)(E)=0\}.$$
It is not clear whether $\cN_\lambda$ is of class $\cC^1$, however we show that it is homeomorphic with the unit sphere in $\cV^+$, hence  $\cN_\lambda$ is an infinite dimensional topological manifold of infinite codimension.  
For any $\lambda\leq 0$ we set
$$c_\lambda=\inf_{\cN_\lambda}J_\lambda.$$
Firstly we present a partial result which does not solve \eqref{eq:main}, but we show the existence of a minimizing sequence for some $\lambda\leq 0$ with a nontrivial weak limit point.

\begin{Th}\label{thm:main} 
Let $(V)$ be satisfied and $-\lambda_\nu< \lambda\leq -\lambda_{\nu-1}$ for some $\nu\geq 1$.  
Then $c_\lambda >0$ and the following statements hold.\\
a) If $\lambda<-\lambda_\nu+S\mu(\Om)^{\frac{2-p}{p}}$, then $c_\lambda < c_0$ and
there is a Palais-Smale sequence $(E_n)\subset \cN_{\lambda}$ such that $J_\lambda(E_n)\to c_\lambda>0$ and $E_n\weakto E_0\neq 0$ in $W^p_0(\curl;\Om)$, where $\mu(\Om)$ is the Lebesgue measure of $\Om$ and
\begin{equation*}
S=\inf_{|v|_p=1,\;v\in\V}\int_{\Om}|\curlop v|^2\,dx.
\end{equation*}
b) The function $(-\lambda_{\nu},-\lambda_{\nu-1}]\ni\lambda\mapsto c_\lambda\in (0,+\infty)$ is non-decreasing, continuous and $c_\lambda\to 0$ as $\lambda\to -\lambda_\nu^- $. If $c_{\mu_1}=c_{\mu_2}$  for some $-\lambda_{\nu}<\mu_1<\mu_2\leq-\lambda_{\nu-1}$, then 
$c_\lambda$ is not attained for $\lambda\in (\mu_1,\mu_2]$.
\end{Th}

Due to the lack of weak-to-weak$^*$ continuity of $J_\lambda'$, we do not know if $E_0$ in Theorem \ref{thm:main} a) is a critical point and it is still an open question whether $c_\lambda$ is attained for some $\lambda\leq 0$ and $2<p\leq 6$ such that $(V)$ holds. Recall that if $\cV$ is compactly embedded into $L^p(\Om,\R^3)$ with $2<p<6$, then by Lemma \ref{eq:weaktoweakSubcritical},  $J_\lambda'$ is weak-to-weak$^*$ continuous on $\cN_\lambda$, and
$c_\lambda$ is attained for every $\lambda\leq 0$; see \cite{BartschMederski1,BartschMederski2}.\\
\indent In order to solve \eqref{eq:main} we assume that $\Om$ is $G=\cO(2)\times 1$-invariant. Then  we may define a subspace $X^{cyl}\subset \V$ of fields of the form \eqref{eq:main}, such that, by the Palais principle of symmetric criticality, critical points of $J_\lambda|_{X^{cyl}}$ are critical points of the free functional $J_\lambda$, hence weak solutions to \eqref{eq:main}. Similarly as above, there is a sequence of eigenvalues 
$$0<\la_1^{cyl}\le \la_2^{cyl}\le\dots\le\la_k^{cyl}\to\infty$$
of $\curlop(\curlop (\cdot))$ in $X^{cyl}$ with the corresponding finite multiplicities $m(\lambda_k^{cyl})\in\N$. Clearly, the set of eigenvalues $\{\lambda_k^{cyl}: k\in\N\}$ is contained in $\{\lambda_k: k\in\N\}$,
and
again, for $\lambda\leq 0$ we find two closed and orthogonal subspaces $\cV^{cyl+}$ and $\tcV^{cyl}$ of $X^{cyl}$ such that the quadratic form $Q$
is positive on $\cV^{cyl+}$ and semi-negative on $\tcV^{cyl}$. Then the {\em symmetric Nehari-Pankov manifold} is given by
\begin{equation}\label{eq:DefOfNcyl}
\cN^{cyl}_\lambda:=\{E\in X^{cyl}\setminus\tcV^{cyl}:\; J_\lambda'(E)|_{\R E\oplus\tcV^{cyl}}=0\},
\end{equation}
and
$$c_\lambda^{cyl}:=\inf_{\cN_\lambda^{cyl}}J_\lambda.$$
Critical points of $J_\lambda$ in $\cN_\lambda^{cyl}$ that attain $c_\lambda^{cyl}$ will be called {\em symmetric ground states}.\\
\indent Our first existence result  reads as follows.
\begin{Th}\label{thm:main2} Let $(V)$ be satisfied, suppose that $\Om$ is $G$-invariant and $-\lambda_\nu^{cyl}< \lambda\leq -\lambda_{\nu-1}^{cyl}$ for some $\nu\geq 1$. Then $c_\lambda^{cyl}>0$ and there is 
$$\eps_\nu\in [S\mu(\Om)^{\frac{2-p}{p}},\lambda_{\nu}^{cyl}]$$ such that the following statements hold.\\
a) If $\lambda\in(-\lambda_{\nu}^{cyl},-\lambda_{\nu}^{cyl}+\eps_\nu)$, then there is a symmetric ground state solution to \eqref{eq:main}, i.e. $c_\lambda^{cyl}$ is attained by a critical point of $J_\lambda$. Moreover $c_\lambda^{cyl}<c_0^{cyl}$.\\
b) If $\eps_\nu<\lambda_\nu^{cyl}-\lambda_{\nu-1}^{cyl}$, then $c_\lambda^{cyl}$ is not attained for $\lambda\in(-\lambda_{\nu}^{cyl}+\eps_n,-\lambda_{\nu-1}^{cyl}]$, and $c_\lambda^{cyl}=c_0^{cyl}$  for $\lambda\in[-\lambda_{\nu}^{cyl}+\eps_n,-\lambda_{\nu-1}^{cyl}]$.\\
c) $c_\lambda^{cyl}\to 0$ as $\lambda\to -\lambda_\nu^{cyl-}$, and the function $$(-\lambda_{\nu}^{cyl},-\lambda_{\nu}^{cyl}+\eps_\nu]\cap (-\lambda_{\nu}^{cyl},-\lambda_{\nu-1}^{cyl}]\ni\lambda\mapsto c_\lambda^{cyl}\in (0,+\infty)$$ is continuous and strictly increasing.
\\
d) If 
$
\tilde{m}(\lambda):=\sharp\big\{k: -\lambda^{cyl}_k <\lambda <-\lambda^{cyl}_k + S\mu(\Om)^{\frac{2-p}{p}}\big\},$
then there are at least $\tilde{m}(\lambda)$ pairs of solutions $E$ and $-E$ to \eqref{eq:main} of the form \eqref{eq:sym}. Moreover, if $\mu_n\to\mu_0$ in $I$ as $n\to\infty$, and $E_n$ is a symmetric ground state of $J_{\mu_n}$ for $n\geq 1$, then passing to a subsequence, $(E_n)$ tends to a symmetric ground state $E_0$ of $J_{\mu_0}$ in the strong topology of $W^p_0(\curl;\Om)$. In particular, the set of symmetric ground states of $J_\lambda$ is compact for $\lambda\in I$.
\end{Th}

Recall that a similar multiplicity result as in Theorem \ref{thm:main2} d) for the Brezis-Nirenberg problem \eqref{eq:BNproblem}  has been obtained by Cerami, Fortunato and Struwe \cite{CeramiAIHP1984}, \cite{Struwe}[Theorem 2.6].\\
\indent In order to deal with problem \eqref{eq:main} and prove Theorem \ref{thm:main} and Theorem \ref{thm:main2} we have to develop a new critical point theory in Section \ref{sec:Nehari}.  In the subcritical case, i.e. when  $\cV$ is compactly embedded into $L^p(\Om,\R^3)$, we know that $J'_\lambda$ is weak-to-weak$^*$ continuous on $\cN_\lambda$ (see Lemma \ref{eq:weaktoweakSubcritical} below), so that Theorem \eqref{eq:main} a) provides existence results for some $\lambda\leq 0$ and which have been obtained for all $\lambda\leq 0$ in \cite{BartschMederski1,BartschMederski2}. However, an additional information in the subcitical case in comparison to \cite{BartschMederski1,BartschMederski2}, is that
the map $(-\lambda_\nu,-\lambda_{\nu-1}]\ni\lambda\mapsto c_\lambda \in (0,+\infty)$ is strictly increasing and continuous and $c_{\lambda}\to 0$ as $\lambda\to -\lambda_\nu^{-}$, which follows from Theorem \ref{ThContinuityuOfClambda} and Lemma \ref{LemEstimateLevels} below.\\
\indent If $\lambda=0$, we do not know whether we get a nonexistence result for  \eqref{eq:main}  on star-shaped domains by means of a Pohozaev-type argument. In case of $\Om=\R^3$, a variant of the Pohozaev identity has been obtained in \cite{MederskiENZ}[Theorem 2.4].
We remark only that in our situation, the Palais-Smale condition cannot be satisfied at a positive level on star-shaped domains $\Om$ when $\lambda=0$ and $p=6$. Indeed, if $E\in \cN_0$, then taking $E_\eps$ given by the construction \eqref{eq:defOfE_eps} in the space $W^6_0(\curl;\Om)$, we easily check that $J'_0(E_\eps)\to0$ and $J_0(E_\eps)=J_0(E)>0$ as $\eps\to 0$. On the other hand, 
$E_\eps\weakto 0$ in $W_0^p(\curl;\Om)$, so that we cannot find a convergent subsequence.

\section{Critical point theory of ground states}\label{sec:Nehari}

Firstly we recall a general setting from \cite{BartschMederski1,BartschMederski2} which allows to work with the Nehari-Pankov manifold. Let $X$ be a reflexive Banach space with norm $\|\cdot\|$ and with a topological direct sum decomposition $X = X^+\oplus \tX$, where $X^+$ is a Hilbert space with a scalar product. For $u \in X$ we denote by $u^+ \in X^+$ and $\tu  \in \tX$ the corresponding summands so that $u = u^++\tu $. We may assume that $\langle u,u \rangle = \|u\|^2$ for any
$u\in X^+$ and that $\|u\|^2 = \|u^+\|^2+\|\tu \|^2$. The topology $\cT$ on $X$ is defined as the product of the norm topology in $X^+$ and the weak topology in $\tX$. Thus $u_n \cTto u$ is equivalent to $u_n^+ \to u^+$ and $\tu_n \weakto \tu $.\\
\indent Let $J \in \cC^1(X,\R)$ be a functional on $X$ of the form
\begin{equation}\label{EqJ}
J(u) = \frac12\|u^+\|^2-I(u) \quad\text{for $u=u^++\tu \in X^+\oplus \tX$}.
\end{equation}
We define the set
\begin{equation}\label{eq:NehariDef}
\cN := \{u\in X\setminus\tX: J'(u)|_{\R u\oplus \tX}=0\}
\end{equation}
and suppose the following assumptions hold:
\begin{itemize}
\item[(A1)] $I\in\cC^1(X,\R)$ and $I(u)\ge I(0)=0$ for any $u\in X$.
\item[(A2)] $I$ is $\cT$-sequentially lower semicontinuous:
    $u_n\cTto u\quad\Longrightarrow\quad \liminf I(u_n)\ge I(u)$
\item[(A3)] If $u_n\cTto u$ and $I(u_n)\to I(u)$ then $u_n\to u$.
\item[(A4)] There exists $r>0$ such that $a:=\inf\limits_{u\in X^+:\|u\|=r} J(u)>0$.
\item[(B1)] $\|u^+\|+I(u)\to\infty$ as $\|u\|\to\infty$.
\item[(B2)] $I(t_nu_n)/t_n^2\to\infty$ if $t_n\to\infty$ and $u_n^+\to u^+$ for some $u^+\neq 0$ as $n\to\infty$.
\item[(B3)] $\frac{t^2-1}{2}I'(u)(u) + tI'(u)(v) + I(u) - I(tu+v) < 0$ for every $u\in \cN$, $t\ge 0$, $v\in \tX$ such that $u\neq tu+v$.
\end{itemize}

\begin{Prop}[see \cite{BartschMederski1,BartschMederski2}]\label{prop:nehari}
For every 
$$u \in S(X^+) := \{u\in X^+:\|u\|=1\}$$ the functional $J$ constrained to
$\R u\oplus\tX = \{tu+v:t\ge0,\ v\in\tX\}$ has precisely two critical points $u_1,u_2$ with positive energy. These are of the form $u_1=t_1u+v_1$, $u_2=t_2u+v_2$ with $t_1>0>t_2$, $v_1,v_2\in\tX$. Moreover, $u_1$ is the unique global maximum of $J|_{\R^+u\oplus\tX}$, and $u_2$ is the unique global maximum of $J|_{\R^-u\oplus\tX}$, where $\R^+=[0,+\infty)$ and $\R^-=(-\infty,0]$. Moreover, $u_1$ and $u_2$ depend continuously on $u\in S(X^+)$.
\end{Prop}

For $u\in S(X^+)$ we set $n(u):=u_1$ with $u_1$ from Proposition~\ref{prop:nehari}. Observe that $n(-u)=u_2$ and
\begin{equation}\label{eq:NehariCharact}
\cN = \{u\in X\setminus\tX: J'(u)|_{\R u\oplus\tX}=0\} = \{n(u): u\in S(X^+)\},
\end{equation}
in particular, $\cN$ is a topological manifold, the so-called {\em Nehari-Pankov manifold}. Note that all critical points of $J$ from $X\setminus \tX$ lie in $\cN$. One easily checks that if $I'(u)(u)>2I(u)$ for $u\neq 0$, then $\cN$ contains all nontrivial critical points.\\
\indent We say that the functional $J$ satisfies the {\em $(PS)_\beta^\cT$-condition in} $\cN$ if every $(PS)_\beta$-sequence $(u_n)_n$ for the unconstrained functional and such that $u_n\in\cN$ has a subsequence which converges in the $\cT$-topology:
\[
u_n \in \cN,\ J'(u_n) \to 0,\ J(u_n) \to \beta \qquad\Longrightarrow\qquad
u_n \cTto u\in X\ \text{ along a subsequence.}
\]
According to \cite{BartschMederski1} (cf. \cite{BartschMederski2}) we know that  the above conditions (A1)-(A4), (B1)-(B3) imply 
\begin{equation}\label{eq:groundstlevel}
c = \inf_\cN J\geq a>0.
\end{equation}
Moreover if $J$ is {\em coercive on $\cN$}, i.e. $J(u)\to \infty$ as $u\in\cN$ and $\|u\|\to\infty$, and satisfies the $(PS)_{c}^\cT$-condition in $\cN$,  then $c$ is achieved by a critical point of $J$,  and if $J$ is additionally even and  satisfies the $(PS)_{\beta}^\cT$-condition in $\cN$ for all $\beta\in\R$, then $J$ has an unbounded sequence of critical values.\\ 
\indent Since  the $(PS)_\beta^\cT$-condition in $\cN$ may be not satisfied for all $\beta$ or this condition could be difficult to check, we introduce a {\em compactly perturbed} problem with respect to another decomposition of $X$. Namely, suppose that 
$$\tX=X^0\oplus X^1,$$
where $X^0$, $X^1$ are closed in $\tX$, and $X^0$ is a Hilbert space. For $u\in\tX$ we denote $u^0\in X^0$ and $u^1\in X^1$ the corresponding summands so that $u=u^0+u^1$. We use the same notation for the scalar product in $X^+\oplus X^0$ and  $\langle u,u \rangle = \|u\|^2=\|u^+\|^2+\|u^0\|^2$ for any
$u=u^++u^0\in X^+\oplus X^0$, hence $X^+$ and $X^0$ are orthogonal.
We consider another functional $J_{cp}\in\cC^1(X,\R)$ of the form
\begin{equation}\label{EqJ2}
J_{cp}(u) = \frac12\|u^++u^0\|^2-I_{cp}(u) \quad\text{for $u=u^++u^0+u^1 \in X^+\oplus X^0\oplus X^1$.}
\end{equation}
We define the corresponding Nehari-Pankov manifold for $J_{cp}$
\[
\cN_{cp} := \{u\in X\setminus X^1: J_{cp}'(u)|_{\R u\oplus X^1}=0\},
\]
and we assume that $J_{cp}$ satisfies all corresponding assumptions  (A1)-(A4), (B1)-(B3), where we replace $X^+\oplus X^0$, $X^1$ and $I_{cp}$ instead of $X^+$, $\tX$ and $I$ respectively.
Moreover we enlist new additional conditions:
\begin{itemize}
\item[(C1)] $J_{cp}(u_n)-J(u_n)\to 0$  if $(u_n)\subset \cN_{cp}$ is bounded and $(u_n^++u_n^0)\weakto 0$. Moreover there is $M> 0$ such that $J_{cp}(u)-J(u)\leq M \|u^++u^0\|^2$ for $u\in \cN_{cp}$.
\item[(C2)] $I(t_nu_n)/t_n^2\to\infty$ if $t_n\to\infty$ and $(I(tu_n^+))_n$ is bounded away from $0$ for any $t>1$.
\item[(C3)] $J'$ is weak-to-weak$^*$ continuous on $\cN$, i.e. if $(u_n)_n\subset\cN$, $u_n\weakto u$, then  $J'(u_n)\overset{\ast}{\weakto} J'(u)$ in $X^*$. Moreover $J$ is weakly sequentially lower semicontinuous on $\cN$, i.e. if $(u_n)_n\subset\cN$, $u_n\weakto u$ and $u\in\cN$, then $\liminf_{n\to\infty}J(u_n)\geq J(u)$.

\end{itemize}
Observe that (C1) joins two functionals $J_{cp}$ and $J$ such that any bounded sequence of $\cN_{cp}$ with the weakly convergent part to $0$ in $X^+\oplus X^0$ is mapped by $J_{cp}-J$ in a compact set. Therefore, if (C1) holds, then we say that $J_{cp}$ is {\em compactly perturbed} with respect to $J$. It is easy to check that if $I(u)>0$ for $u\in X^+\setminus\{0\}$, then (C2) implies (B2).
Note that in (C3) we require an additional regularity for $J$.\\
\indent If $J$ is coercive on $\cN$, then for any  $(PS)_c$-sequence $(u_n)_n\subset\cN$ 
one has $u_n \weakto u$ up to a subsequence, and the first main difficulty is to ensure that $u$ is nontrivial. The second one it to show that $u$ is a critical point. To this aim, and to analyse the multiplicity of critical points we present the main result of this section.

\begin{Th}\label{ThLink2}
Let $J \in \cC^1(X,\R)$ be coercive on $\cN$ and let $J_{cp} \in \cC^1(X,\R)$ be coercive on $\cN_{cp}$. Suppose that $J$ and $J_{cp}$ satisfy (A1)-(A4), (B1)-(B3) and set $c = \inf_\cN J$ and 
$d = \inf_{\cN_{cp}} J_{cp}$.  Then the following statements hold:

a)  If (C1)-(C2) hold and $\beta<d$, then any $(PS)_\beta$-sequence in $\cN$ contains a weakly convergent subsequence with a nontrivial limit point.

b) If (C1)-(C3) hold and $c<d$, then $c$ is achieved by a critical point (ground state) of $J$.

c) Suppose that $J$ is even and satisfies the $(PS)_\beta^\cT$-condition in $\cN$ for any $\beta<\beta_0$ for some fixed $\beta_0\in (c,+\infty]$. Let
\begin{equation}\label{eq:DefOfm(N,beta)}
m(\cN,\beta_0)=\sup\{\gamma(J^{-1}((0,\beta])\cap \cN):  \beta<\beta_0\}\in \N_0,
\end{equation}
where $\gamma$ stands for the Krasnoselskii genus for closed and symmetric subsets of $X$. 
 Then $J$ has at least $m(\cN,\beta_0)$ pairs of critical points $u$ and $-u$ such that $u\neq 0$ and $c\leq J(u)<\beta_0$.
\end{Th}

\begin{proof}
Similarly as in Proposition \ref{prop:nehari} we define a homeomorphism $n_{cp}:S(X^+\oplus X^0)\to\cN_{cp}$, where $S(X^+\oplus X^0)$ stands for the unit sphere in $X^+\oplus X^0$.
As in \cite{BartschMederski1,SzulkinWeth}  one proves that
\begin{itemize}
\item[(i)] $n:S(X^+)\to\cN$, $n_{cp}:S(X^+\oplus X^0)\to\cN_{cp}$ are homeomorphisms with inverses $\cN \to S(X^+)$, $\cN_{cp} \to S(X^+\oplus X^0)$ given by $u \mapsto u^+/\|u^+\|$, $u \mapsto (u^++u^0)/\|u^++u^0\|$ respectively.

\item[(ii)] $J\circ n:S(X^+)\to\R$ 
 is $\cC^1$.
\item[(iii)] $(J\circ n)'(u) = \|n(u)^+\|\cdot J'(u)|_{T_uS(X^+)}:T_uS(X^+)\to\R$ 
for every $u\in S(X^+)$.
\item[(iv)] $(u_n)_n\subset S(X^+)$ is a Palais-Smale sequence for $J\circ n$ if, and only if, $(n(u_n))_n$ is a Palais-Smale sequence for $J$ in $\cN$. 
\item[(v)] $u\in S(X^+)$ is a critical point of $J\circ n$ if, and only if, $n(u)$ is a critical point of $J$.
\item[(vi)] If $J$  is even, then so is $J\circ n$.
\end{itemize}

a) 
Let $(n(u_n))_n$ be a $(PS)_{\beta}$-sequence for $J$ in $\cN$, $\beta<d$. Observe that by \eqref{eq:groundstlevel} we have $\beta\geq c>0$. Since $X$ is reflexive and $J$ is coercive on $\cN$, then we may assume that $n(u_n)\weakto u$ for some $u\in X$. Suppose that $u=0$. Since $n_{cp}(n(u_n)^+/\|n(u_n)^+\|)$  is  the unique global maximum of $J_{cp}|_{\R^+n(u_n)^+\oplus X^1}$ we find $t_n>0$ and $v_n\in X^1$ such that
$$n_{cp}(n(u_n)^+/\|n(u_n)^+\|)=t_n (n(u_n)^++v_n)\in\cN_{cp}.$$
Suppose that $t_n\to\infty$. In view of (B3) applied to $J_{cp}$ with $t=0$, $v=0$ and $u=t_n(n(u_n)^+ + v_n)$, one obtains
\begin{eqnarray*}
\|n(u_n)^+\|^2= \frac{I_{cp}'(t_n(n(u_n)^+ + v_n))(t_n(n(u_n)^+ + v_n))}{t_n^2} \geq 
\frac{2I_{cp}(t_n(n(u_n)^+ + v_n))}{t_n^2}. 
\end{eqnarray*}
Since (C1) holds, we have
\begin{eqnarray*}
I(t_n(n(u_n)^+ + v_n))-I_{cp}(t_n(n(u_n)^+ + v_n))&=&
J_{cp}(t_n(n(u_n)^+ + v_n))-J(t_n(n(u_n)^+ + v_n))\\
&\leq &Mt_n^2\|n(u_n)^+\|^2,
\end{eqnarray*}
and we get 
\begin{eqnarray*}
(1+2M)\|n(u_n)^+\|^2\geq 
\frac{2I(t_n(n(u_n)^+ + v_n))}{t_n^2}.
\end{eqnarray*}
Then, by (C2) there is $t>1$ such that
\begin{equation}\label{eq:eqlimI1}
I(tn(u_n)^+)\to 0
\end{equation}
as $n\to\infty$. Since $n(u_n)\in\cN$, and in view of (B3), we get
$$\frac{t^2-1}{2}\|n(u_n)^+\|^2=\frac{t^2-1}{2}I'(n(u_n))(n(u_n))\leq I(tn(u_n)^+)-I(n(u_n)).$$
Then, by (A1) and \eqref{eq:eqlimI1} we infer that
$$n(u_n)^+\to 0.$$
Observe that 
$$0<\beta=\lim_{n\to\infty}J(n(u_n))=\lim_{n\to\infty}\Big(\frac{1}{2}\|n(u_n)^+\|^2-I(n(u_n))\Big)\leq \lim_{n\to\infty}\frac{1}{2}\|n(u_n)^+\|^2=0$$
and we get a contradiction. Therefore $t_n$ is bounded 
and passing to a subsequence 
\begin{eqnarray*}
n_{cp}(n(u_n)^+/\|n(u_n)^+\|)^++n_{cp}(n(u_n)^+/\|n(u_n)^+\|)^0&=&n_{cp}(n(u_n)^+/\|n(u_n)^+\|)^+\\
&=&t_nn(u_n)^+\weakto 0\hbox{ in }X.
\end{eqnarray*}
Since 
$$J_{cp}(n_{cp}(n(u_n)^+/\|n(u_n)^+\|))\leq \frac{1}{2} \|t_n n(u_n)^+\|^2,$$
then $(n_{cp}(n(u_n)^+/\|n(u_n)^+\|))_n$ is bounded and by (C1) we get
\begin{eqnarray*}
\beta+o(1)&=&J(n(u_n))\\
&\geq& J(n_{cp}(n(u_n)^+/\|n(u_n)^+\|))\\
&\geq& d + o(1),
\end{eqnarray*}
which gives a contradiction with $\beta<d$. Thus $u\neq 0$. 

b) The existence of a $(PS)_{c}$-sequence $(n(u_n))_n$ for $J$ in $\cN$ follows from (ii) and (iv) because 
$$c = \inf_{S(X^+)} J\circ n.$$
In view of a) we get $n(u_n)\weakto u\neq 0$ and by (C3)  we have $J'(u)=0$. Thus $u\in\cN$ and
$$c=\liminf_{n\to\infty}J(n(u_n))\geq J(u)\geq c.$$

c) Let $K\subset S(X^+)$ be the set of all critical points of $\Phi:=J\circ n:S(X^+)\to \R$ and let $m(\cN,\beta_0)\geq 1$.
If $K$ is infinite, then by (v) we have that $n(K)$ is an infinite set of critical points of $J$ and we conclude. Suppose that $K$ is finite.
If $(u_n)\subset S(X^+)$ is a Palais-Smale sequence of $\Phi$, i.e. $\Phi(u_n)\to \beta <\beta_0$ and $\Phi'(u_n)\to 0$, then $(n(u_n))$ is a $(PS)_{\beta}$-sequence for $J$ and 
$n(u_n)\cTto u_0$ for some $u_0\in\cN$ along a subsequence. Hence $u_n\to n^{-1}(u_0)\in S(X^+)$. Therefore $\Phi$ satisfies the usual Palais-Smale condition and by means of the standard methods (e.g. \cite{Rabinowitz:1986}[Theorem 8.10]) we show that the Lusternik-Schnirelmann values
\begin{eqnarray*}
\beta_k&:=& \inf\{\beta\in\R:  \gamma(\Phi^{\beta})\geq k\}.
\end{eqnarray*}
are critical values and
satisfy
$$\beta_1<\beta_2<...<\beta_{m(\cN,\beta_0)}.$$
\end{proof}

\begin{Cor}
Suppose that the assumptions of Theorem \ref{ThLink2} are satisfied and (C1)-(C3) hold. If $X^0=\{0\}$, $d$ is achieved by a critical point and $J(u)<J_{cp}(u)$ for any $u\in \cN$, then $c<d$ and $c$ is achieved by a critical point.
\end{Cor}
\begin{proof}
Let $J'_{cp}(u_0)=0$ and $J_{cp}(u_0)=d$. Observe that $n(u_0)\in \R^+u_0\oplus \tX=\R^+u_0\oplus X^1$ and
$$c\leq J(n(u_0))<J_{cp}(n(u_0))\leq J_{cp}(u_0)=d.$$
By Theorem \ref{ThLink2} b) we conclude.
\end{proof}

\begin{Th}\label{ThContinuityuOfClambda}
Suppose that there is a family of functionals $J_\lambda\in\cC^1(X,\R)$ of the form
$$J_\lambda(u) = \frac12\|u^+\|^2-I_\lambda(u) \quad\text{for $u=u^++\tu \in X^+\oplus \tX$},$$
where $J_\lambda$ and $I_\lambda$ satisfy assumptions (A1)-(A4), (B1)-(B3) and (C2) for $\lambda\in I\subset \R$, $\cN_\lambda$ is given by \eqref{eq:NehariDef} and $J_\lambda$ is coercive on $\cN_\lambda$. Suppose that there is $L>0$ such that
\begin{equation}\label{eq:thastract2eq1}
0< J_{\lambda_1}(u)-J_{\lambda_2}(u)\leq L(\lambda_1-\lambda_2)\|u\|^2
\hbox{ for }u\in X\setminus\tX\hbox{ and } \lambda_1>\lambda_2,\; \lambda_1,\lambda_2\in I.
\end{equation}
Then the map $I\ni\lambda\mapsto c_{\lambda}\in (0,+\infty)$ is non-decreasing and continuous. 
Moreover, if $c_\lambda:=\inf_{\cN_\lambda}J_\lambda$ is attained for every $\lambda\in I$, then the map $I\ni\lambda\mapsto c_{\lambda}\in (0,+\infty)$ is  strictly increasing. 
\end{Th}
\begin{proof}
Observe that for $\lambda_1>\lambda_2$ and for any $u\in \cN_{\lambda_1}$ we have
\begin{equation}\label{eq:Abstractest1}
 J_{\lambda_1}(u) \geq J_{\lambda_1}(n_{\lambda_2}(u))
>J_{\lambda_2}(n_{\lambda_2}(u))\geq c_{\lambda_2},
\end{equation}
where  $n_\lambda:S(X^+)\to\cN_\lambda$ for $\lambda\in I$ denotes the homeomorhism defined after Proposition \ref{prop:nehari}. Hence $c_{\lambda_1}\geq c_{\lambda_2}$
and  the map $I\ni\lambda\mapsto c_{\lambda}\in (0,+\infty)$ is non-decreasing. 
Now suppose that $\lambda_n\in I$
for $n\geq 0$ and let $\lambda_n\to\lambda_0$ in $I$. Take $u_n\in\cN_{\lambda_n}$ such that 
$$c_{\lambda_n}+\frac{1}{n}\geq J_{\lambda_n}(u_n).$$ 
Let $\mu=\inf_{n\geq 0}\{\lambda_n\}$  and $\nu=\sup_{n\geq 0}\{\lambda_n\}$. Then $\mu,\nu\in I$ and by \eqref{eq:thastract2eq1} for any $u\in\cN_{\nu}$
$$J_{\mu}(u_n)\leq J_{\lambda_n}(u_n)\leq c_n+\frac{1}{n}\leq J_{\lambda_n}(n_{\lambda_n}(u))+\frac{1}{n}
\leq J_{\nu}(n_{\lambda_n}(u))+\frac{1}{n}\leq J_{\nu}(u)+\frac{1}{n}.$$
Since $J_\mu$ is coercive, we get that $(u_n)$ is bounded. 
Let $n_{\lambda_0}(u_{n})=t_n(u_{n}^++v_n)\in \cN_{\lambda_0}$ and suppose that $t_n\to\infty$.
In view of (B3) applied to $I_{\lambda_0}$ one obtains
\begin{eqnarray*}
\|u_{n}^+\|^2= \frac{I_{\lambda_0}'(t_n(u_n^+ + v_n))(t_n (u_n^+ + v_n))}{t_n^2} \geq 
\frac{2I_{\lambda_0}(t_n(u_n^+ + v_n))}{t_n^2},
\end{eqnarray*}
and by (C2) applied to $I_{\lambda_0}$, we infer that 
there is $t>1$ such that
\begin{equation*}
I_{\lambda_0}(tu_n^+)\to 0
\end{equation*}
as $n\to\infty$. In view of \eqref{eq:thastract2eq1} we get
$$I_{\lambda_n}(tu_n^+)\to 0,$$
and by (B3) we obtain
$$\frac{t^2-1}{2}\|u_n^+\|^2=\frac{t^2-1}{2}I'_{\lambda_n}(u_n)(u_n)\leq I_{\lambda_n}(tu_n^+)-I_{\lambda_n}(u_n)\leq I_{\lambda_n}(tu_n^+)\to 0.$$
Thus $$u_n^+\to 0.$$
Observe that 
\begin{eqnarray*}
0&<&c_{\mu}\leq \lim_{n\to\infty}J_{\mu}(n_{\mu}(u_n))\leq\lim_{n\to\infty}J_{\lambda_n}(n_{\mu}(u_n))\leq \lim_{n\to\infty}J_{\lambda_n}(u_n)\\
&=&\lim_{n\to\infty}\Big(\frac{1}{2}\|u_n^+\|^2-I_{\lambda_n}(u_n)\Big)\leq \lim_{n\to\infty}\frac{1}{2}\|u_n^+\|^2=0
\end{eqnarray*}
and we get a contradiction. Therefore $t_n$ is bounded.
Since 
\begin{eqnarray*}
\|t_nu_{n}^+\|^2 \geq 
2I_{\lambda_0}(n_{\lambda_0}(u_n)),
\end{eqnarray*}
then by (B1) we get the boundedness of $(n_{\lambda_0}(u_n))$.
Observe that
\begin{eqnarray*}
c_{\lambda_0}&\leq &J_{\lambda_0}(n_{\lambda_0}(u_{n}))\leq J_{\lambda_n}(n_{\lambda_0}(u_{n}))
+L|\lambda_n-\lambda_0|\|n_{\lambda_0}(u_{n})\|^2\\
&\leq& J_{\lambda_n}(u_{n})
+L|\lambda_n-\lambda_0|\|n_{\lambda_0}(u_{n})\|^2\\
&\leq& c_{\lambda_n}+\frac{1}{n}+L|\lambda_n-\lambda_0|\|n_{\lambda_0}(u_{n})\|^2,
\end{eqnarray*} 
thus
\begin{equation}\label{eq:proofCont1}
c_{\lambda_0}\leq \liminf_{n\to\infty} c_{\lambda_n}.
\end{equation}
Now let $u_0\in \cN_0$ and assume that $n_{\lambda_n}(u_0)=t_n (u_0^++v_n)\in \cN_{\lambda_n}$ for some $t_n>0$ and $v_n\in \tX$. Then, in view of (B3) applied to $I_{\lambda_n}$, one obtains
\begin{eqnarray*}
\|u_{0}^+\|^2= \frac{I_{\lambda_n}'(t_n(u_0^+ + v_n))(t_n (u_0^+ + v_n))}{t_n^2} \geq 
\frac{2I_{\lambda_n}(t_n(u_0^+ + v_n))}{t_n^2} 
\geq 
\frac{2I_{\nu}(t_n(u_0^+ + v_n))}{t_n^2}.
\end{eqnarray*}
Then by (B2) applied to $I_{\nu}$, we infer that $t_n$ is bounded. Since 
\begin{eqnarray*}
\|t_nu_{0}^+\|^2 \geq 
2I_{\nu}(n_{\lambda_n}(u_0)),
\end{eqnarray*}
then by (B1) we get the boundedness of $(n_{\lambda_n}(u_0))$.
Moreover
\begin{eqnarray*}
J_{\lambda_0}(u_{0})&\geq& J_{\lambda_0}(n_{\lambda_n}(u_{0}))\geq J_{\lambda_n}(n_{\lambda_n}(u_{0}))
-L|\lambda_0-\lambda_n|\|n_{\lambda_n}(u_{0})\|^2\\
&\geq& c_{\lambda_n}-L|\lambda_0-\lambda_n|\|n_{\lambda_n}(u_{0})\|^2.
\end{eqnarray*} 
Hence
$c_{\lambda_0}\geq \limsup_{n\to\infty} c_{\lambda_n}$ and taking into account
\eqref{eq:proofCont1} we get
$$c_{\lambda_0}=\lim_{n\to\infty} c_{\lambda_n},$$
which completes the proof of the continuity of $c_\lambda$. Now observe that
as in \eqref{eq:Abstractest1} we get the strict inequality $c_{\lambda_1}>c_{\lambda_2}$  provided that $c_{\lambda_1}$ is attained. 
\end{proof}

\section{Problem \eqref{eq:main} and proofs}\label{sec:varsetting}
Recall that there is a continuous tangential trace operator $\gamma_t:H(\curl;\Om)\to H^{-1/2}(\pa\Om)$ such that
$$
\gamma_t(E)=\nu\times E_{|\pa\Om}\qquad\text{for any $E\in \cC^{\infty}(\overline\Om,\R^3)$}
$$
and (see \cite[Theorem~3.33]{Monk})
$$
H_0(\curl;\Om)=\{E\in H(\curl;\Om): \gamma_t(E)=0\},
$$
so that any vector field $E\in  W^p_0(\curl;\Om)=\V\oplus\W\subset H_0(\curl;\Om)$ satisfies the boundary condition \eqref{eq:bc}.\\
\indent The spectrum of the curl-curl operator in $H_0(\curl;\Om)$ consists of the eigenvalue $0$ with infinite multiplicity and the eigenspace $\W$, and of a sequence of eigenvalues $0<\la_1\le \la_2\le\dots\le\la_k\to\infty$ with finite multiplicities $m(\lambda_k)$ and eigenfunctions in $\cV$. Indeed, similarly as in \cite[Theorem~4.18]{Monk}, for any $f\in L^2(\Om,\R^3)$ the equation
\begin{equation}\label{eq:operator}
\curlop(\curlop v) + v = f
\end{equation}
has a unique solution $v\in\cV$ and the operator
\[
K:L^2(\Om,\R^3) \to \cV\subset L^2(\Om,\R^3),\quad
 Kf=v\text{ solves \eqref{eq:operator},}
\]
is self-adjoint and compact, since $\cV\subset X_N$ embeds compactly into $L^2(\Om,\R^3)$ due to \cite{Amrouche}[Theorem 2.8].\\
\indent In this section we assume (V), and as a consequence of the compact embedding $\cV\subset L^2(\Om,\R^3)$
\[
(E_1,E_2) = \int_{\Om}\langle \curlop E_1,\curlop E_2\rangle\,dx
\]
in $\cV$ is equivalent to the standard inner product in $H(\curl;\Om)$
$$\langle E_1,E_2\rangle = \int_{\Om}\langle \curlop E_1,\curlop E_2\rangle+ \langle E_1, E_2\rangle\,dx,$$
cf. \cite{Monk}[Corollary 3.51].
For $v\in\cV$ and $w\in\cW$ one has
\begin{equation*}
\int_{\Om}\langle v,w\rangle \,dx = 0,
\end{equation*}
which means that $\cV$ and $\cW$ are orthogonal in $L^2(\Om,\R^3)$. Moreover in $W_0^p(\curl;\Om)=\V\oplus\W$ we consider the following norm
$$\|v+w\|:=\big((v,v)+|w|_p^2\big)^{\frac12}\hbox{ for }v+w\in\V\oplus\W,$$
which is equivalent with $\|\cdot\|_{W_0^p(\curl;\Om)}$ due to the embedding $\cV\subset L^p(\Om,\R^3)$.
Clearly, $\cW$ contains all gradient vector fields 
$$\nabla W^{1,p}(\Om)\subset\cW.$$ 
Recall that if $\Omega$ is simply connected with connected boundary, then $\nabla W^{1,p}(\Om)=\cW$ as in \cite{BartschMederski1}. However, for general domains $\{w\in \W: \div(w)=0\}$ may be nontrivial and $\nabla W^{1,p}(\Om)\subsetneq\cW$; see \cite{BartschMederski2} for problem \eqref{eq:main} on general domains with subcritical nonlinearities, where $\cV$ is compactly embedded into $L^p(\Om,\R^3)$.\\
\indent Let $\lambda\leq 0$ and  
\[
\nu:=\min\{k\in\N:\la_{k}+\lambda>0\} = \max\{k\in\N:\la_{k-1}+\lambda\le0\}
\]
be the dimension of the semi-negative eigenspace, where $\lambda_0=0$. Then, the quadratic form $Q$ given by \eqref{eq:DefQ} is 
negative semidefinite on
$\tcV$
and negative definite if $\lambda_{\nu-1}<-\lambda$. Here  $\tcV$ is the finite sum of the eigenspaces associated to all $\la_k$ for $k<\nu$ and $\tcV=\{0\}$ if $\nu=1$. Moreover $Q$ is
positive definite on the space $\cV^+= \tcV^{\perp}$ being the infinite sum of the eigenspaces associated to the eigenvalues $\la_k$ for $k\geq \nu$.
For any $v\in\V$ we denote $v^+\in\V^+$ and $\tv \in\tcV$ the corresponding summands such that $v=v^++\tv $.\\
\indent Let $X^+:=\V^+$, $\tX:=\tcV\oplus\W$ and we consider the functional $J:X=\cV\oplus\cW\to\R$ defined by \eqref{eq:defOfJ}, i.e. for $E=v+w=v^++\tv+w$ we have
\[
\begin{aligned}
J_\lambda(E)
 &= \frac12\int_{\Om}|\curlop v|^2\,dx + \frac\la2\int_{\Om} |v+w|^2\,dx
    - \frac1p\int_{\Om}|v+w|^p\,dx\\
 &= \frac12\|v\|^2 + \frac\la2\int_{\Om}(|v|^2+|w|^2)\,dx
    - \frac1p\int_{\Om}|v+w|^p\,dx\\
    &=\frac12\|v^+\|^2 - I_\lambda(v+w),
\end{aligned}
\]
where
\begin{equation}\label{eq:DefOfI}
I_\lambda(v+w) = -\frac12\|\tv\|^2 - \frac\la2\int_\Om\left(|v|^2+|w|^2\right)\,dx
          + \frac1p\int_\Om |v+w|^p\,dx.
\end{equation}
Hence  $J_\lambda$ has the form \eqref{EqJ} and
we shall show that $J_\lambda$ satisfies the assumptions (A1)-(A4), (B1)-(B3) and (C1)-(C2) from Section \ref{sec:Nehari}.
\begin{Lem}\label{eq:LemB3check}
If $E\in \V\oplus \W$, $\tv\in\tcV$, $w\in \W$ and $t\geq 0$, then
\begin{equation}\label{eq:LemB3checkeq}
J_\lambda(E)\geq J_\lambda(tE+\tv+w) -J'_\lambda(E)\left(\frac{t^2-1}{2}E+t(\tv+w)\right).
\end{equation}
Moreover the strict inequality holds provided that $E\neq tE+\tv+w$.
\end{Lem}
\begin{proof}
Let $E\in \V\oplus \W$, $\tv\in\tcV$, $w\in \W$ and $t\geq 0$. Then we need to show that
\begin{equation}\label{eq:B3check}
\begin{aligned}
&J_\lambda'(E)\left[\frac{t^2-1}{2}E+t(\tv+w)\right] + J_\lambda(E) - J_\lambda(tE+\tv+w)\\
&\hspace{1cm}
 = -\frac12Q(\tv) - \frac\la2|w|_2^2
     + \int_\Om\vp(t,x)\,dx
 \geq 0
\end{aligned}
\end{equation}
where
\begin{equation*}
\vp(t,x)
 = -\Big\langle |E|^{p-2}E,\frac{t^2-1}{2}E+t(\tv+w)\Big\rangle - \frac1p|E|^p + \frac1p|tE+\tv+w|^p.
\end{equation*}
Let $E(x)\neq 0$. We can check that $\vp(0,x)\geq 0$, $\vp(t,x)\to\infty$ as $t\to\infty$ and note that if $\partial_t\vp (t_0,x)=0$ for some $t_0>0$, then $t_0E+\tv+w=0$ or $|E|^p=|t_0E+\tv+w|^p$, hence $\vp(t_0,x)\geq 0$. Then
we infer that $\vp(t,x)\geq 0$ for any $t\geq0$. If $Q(\tv) < 0$ or $w\neq 0$ then the inequality \eqref{eq:B3check} is strict. 
If $\tv=0$ and $w=0$, then $\vp(t,x)=\big(\frac{t^p}{p}-\frac{t^2}{2}+\frac12-\frac1p\big)|E|^p>0$ provided that $E\neq tE$.
\end{proof}

Similarly as in \cite{BartschMederski1,BartschMederski2} we prove the following lemma.
\begin{Lem}\label{eq:LemA1A4}
Conditions (A1)-(A4), (B1), (B3) hold for $J_\lambda$. 
\end{Lem}

\begin{proof}
As in \cite{BartschMederski1}[Lemma 5.1] we show that (A1)-(A2), (A4) and (B1) hold, i.e.
\begin{itemize}
 \item 
 $I_\lambda$ is of class $\cC^1$,  $I_\lambda(E)\geq I_\lambda(0)=0$ for any $E\in X=\V\times\W$, and since $I_\lambda$ is convex, then $I_\lambda$ is $\cT$-sequentially lower semicontinuous.
\item There is $r>0$ such that
$\displaystyle 0<\inf_{\stackrel{v\in \V^+}{\|v\|=r}}J_\lambda(v)$.
\item $\|v^+\| + I_\lambda(v+ w)\to\infty$ as $\|v+w\|\to\infty$.
\end{itemize}
Moreover we easily check (A3), since $E_n\weakto E_0$ in $L^p(\Om,\R^3)$ and $I_\lambda(E_n)\to I_\lambda(E_0)$ imply $|E_n|_p\to |E_0|_p$ as $n\to\infty$, thus $E_n\to E_0$ in $L^p(\Om,\R^3)$.  Observe that (B3) is a direct consequence of Lemma \ref{eq:LemB3check}.
\end{proof}

\begin{Lem}\label{eq:LemB2C2}
Conditions (B2) and (C2) are satisfied.
\end{Lem}
\begin{proof}
Since $I_\lambda(v)>0$ for $v\in \V^+$, it is enough to show only (C2).
Let us consider sequences $t_n\to\infty$ and $v_n\in\cV$, $w_n\in\cW$ such that
for some $c>0$ and $t>1$ we have
\begin{equation}\label{eq:LemC2}
I_\lambda(tv_n^+)\geq c>0\hbox{ for any }n\geq 1.
\end{equation}
Note that
$$I_\lambda(t_n(v_n+w_n))\geq \frac{t_n^{p-2}}{p}|v_n+w_n|_p^p$$
and if $\liminf_{n\to\infty}|v_n+w_n|_p>0$, then passing to a subsequence
we conclude. Suppose that, up to a subsequence, $v_n+w_n\to 0$ in $L^p(\Om,\R^3)$. Then $v_n^+\to 0$ in $L^p(\Om,\R^3)$ and in $L^2(\Om,\R^3)$, hence 
$I_\lambda(tv_n^+)\to 0$, which contradicts \eqref{eq:LemC2}.
\end{proof}

\subsection{Weak-to-weak$^*$ continuity}

In general $J'_\lambda$ is not weak-to-weak$^*$ continuous in $\V\oplus\W$, so it is not clear whether a weak limit point of a Palais-Smale sequence is a critical point. We are able to show this continuity of $J_\lambda$ on $\cN_\lambda$ provided that $\cV$ is compactly embedded into $L^p(\Om,\R^3)$. 

\begin{Lem}\label{eq:weaktoweakSubcritical}
If $\cV$ is compactly embedded into $L^p(\Om,\R^3)$, then $J'_\lambda$ is weak-to-weak$^*$ continuous on $\cN_\lambda$ and (C3) holds.
\end{Lem}

\begin{proof}

Let us define $K:L^p(\Om,\R^3)\times \cW\to\R$ given by
$$K(v,w):=-\frac\la2\int_{\Om} |v+w|^2\,dx
    + \frac1p\int_{\Om}|v+w|^p\,dx$$
for $v\in L^p(\Om,\R^3)$ and $w\in\W$. 
Let $\xi:L^p(\Om,\R^3)\to\cW$ be a map such that
\begin{equation}\label{defOfxi}
K(v,\xi(v))=\min_{w\in \cW} K(v,w)
\end{equation}
for $v\in L^p(\Om,\R^3)$. Since $K(v,\cdot)$ is strictly convex and coercive for $v\in\V$, then $\xi$ is well-defined. We show that $\xi$ is continuous. Take $v_n\to v_0$ in  $L^p(\Om,\R^3)$ and since
$$K(v_n,\xi(v_n))\leq K(v_n,0)$$
then $(\xi(v_n))$ is bounded. We may assume that $\xi(v_n)\weakto w_0$ in $L^p(\Om,\R^3)$ and in $L^2(\Om,\R^3)$. By the weak lower continuity we infer that
$$K(v_0,\xi(v_0))=\lim_{n\to\infty} K(v_n,\xi(v_0))
\geq \liminf_{n\to\infty} K(v_n,\xi(v_n))\geq K(v_0,w_0)\geq K(v_0,\xi(v_0)),$$
hence $\xi(v_0)=w_0$. Then $\xi(v_n)\weakto \xi(v_0)$ in $L^p(\Om,\R^3)$ and $|v_n+\xi(v_n)|_p\to |v_0+\xi(v_0)|_p$. Therefore $\xi(v_n)\to \xi(v_0)$ in $\W$.\\
\indent Note that if $v+w\in\cN_\lambda$, then
$$K(v,w)=\frac12\|\tv\|^2+I_\lambda(v+w)\leq 
\frac12\|\tv\|^2+I_\lambda(v+\psi)=K(v,\psi)$$
for any $\psi\in\W$. Hence $w=\xi(v)$.\\ 
\indent Now suppose that $E_n\weakto E_0$ and $E_n=v_n+w_n\in\cN_\lambda$ for $n\geq 1$. Since $\cV$ embeds compactly into $L^p(\Om,\R^3)$ and $\xi$ is continuous, then, passing to a subsequence, we may assume that $E_n=v_n+\xi(v_n)\to E_0$ in $L^p(\Om,\R^3)$.
Now observe that for $\vp+\psi\in\V\oplus\W$
\begin{eqnarray*}
J'_\lambda(E_n)(\vp+\psi)&=& (E_n,\vp) +\lambda
\int_{\Om} \langle E_n,\vp+\psi\rangle\,dx
    - \int_{\Om}\langle |E_n|^{p-2}E_n,\vp+\psi\rangle\,dx\\
    &\to& J'_\lambda(E_0)(\vp+\psi),
\end{eqnarray*}
which completes the proof of the weak-to-weak$^*$ continuity.
Moreover if $E_0\in\cN_\lambda$, then
$$\liminf_{n\to\infty} J_\lambda(E_n)=\liminf_{n\to\infty}\Big(\frac12-\frac1p\Big)
    \int_{\Om} |E_n|^{p}\,dx=J_\lambda(E_0).$$
\end{proof}

In our problem \eqref{eq:main}, however, $\V$ is not compactly embedded into $L^p(\Om,\R^3)$ in general, and we require the cylindrical symmetry to get (C3) in the subspace $X^{cyl}$.
Namely we assume that  $\Om$ is $G$-invariant, and  then $J_\lambda$ is $G$-invariant and any $E\in X^G=\V^G\oplus \W^G$ has a unique decomposition $E=E_\tau+E_\rho+E_\zeta$ with summands of the form
\begin{equation*}
E_\tau(x)
 = \al(r,x_3)\begin{pmatrix}-x_2\\x_1\\0\end{pmatrix},\;
E_\rho(x)
 = \be(r,x_3)\begin{pmatrix}x_1\\x_2\\0\end{pmatrix},\;
E_\zeta(x)
 = \ga(r,x_3)\begin{pmatrix}0\\0\\1\end{pmatrix},
\end{equation*}
where $r=\sqrt{x_1^2+x_2^2}$ and $X^G$ consists of all $G$-equivariant vector fields. 
Then we may define the following linear isometry 
\[
S:X^G \to X^G, \quad S(E_\tau+E_\rho+E_\zeta) := E_\tau-E_\rho-E_\zeta
\]
and $J_\lambda$ is invariant under $S$. Similarly as in the subcritical case \cite{BartschMederski1,BartschMederski2},  by the principle of symmetric criticality it is sufficient to look for critical points of $J_\lambda$ constrained to
\[
X^{cyl}:=(X^G)^{S} = \{E\in X^G: S(E)=E\} = \{E\in X^G: E=E_\tau\}\subset\cV.
\]
As above, we find $\nu\geq 0$ such that $-\lambda_\nu^{cyl}<\lambda\leq -\lambda_{\nu-1}^{cyl}$,  $Q$ is positive on $\V^{cyl+}\subset X^{cyl}$ and semi-negative on $\tcV^{cyl}$, where $\cV^{cyl+}= (\tcV^{cyl})^{\perp}$ is the infinite sum of the eigenspaces associated to the eigenvalues $\la_k^{cyl}$ for $k\geq \nu$ and 
 $\tcV^{cyl}$ is the finite sum of the eigenspaces associated to all $\la_k$ for $k<\nu$ and $\tcV^{cyl}=\{0\}$ if $\nu=1$.
Since $\V^{cyl+}\subset \V^{+}$ and $\tcV^{cyl}\subset \tcV$, any $v\in\V^{cyl}$ has the corresponding decomposition $v=v^++\tv $.\\
\indent Note that $J_\lambda|_{X^{cyl}}:X:=X^{cyl}=\V^{cyl+}\oplus \tcV^{cyl}\to\R$ satisfies conditions (A1)-(A4), (B1)-(B3) and (C2) similarly as in Lemma \ref{eq:LemA1A4} and Lemma \ref{eq:LemB2C2}, where $\cN_\lambda^{cyl}$ is given by \eqref{eq:DefOfNcyl},
$X^+:=\V^{cyl+}$, $\tX:=\tcV^{cyl}$ and for $E=v^++\tv\in \V^{cyl+}\oplus \tcV^{cyl}$ we have
\[
\begin{aligned}
J_\lambda(E)
 &= \frac12\int_{\Om}|\curlop v|^2\,dx + \frac\la2\int_{\Om} |v|^2\,dx
    - \frac1p\int_{\Om}|v|^p\,dx\\
    &=\frac12\|v^+\|^2 - I_\lambda(v).
\end{aligned}
\]

\begin{Lem}\label{eq:weaktoweakCritical}
$J'_\lambda$ is weak-to-weak$^*$ continuous in $\V$, hence on $\cN_\lambda^{cyl}$ and (C3) holds for $J_\lambda|_{X^{cyl}}$.
\end{Lem}

\begin{proof}
\indent Suppose that $E_n\weakto E_0$ in $\V$. Since $\cV$ embeds compactly into $L^{2}(\Om,\R^3)$, and in view of (V), and $\cV$ embeds continuously $L^{p}(\Om,\R^3)$, then, passing to a subsequence, we may assume that $E_n\to E_0$ in  $L^{2}(\Om,\R^3)$ and $E_n(x)\to E_0(x)$ for a.e. $x\in\Om$.
Now observe that for $\vp\in\V$, the family $(\langle |E_n|^{p-2}E_n,\vp\rangle)$ is uniformly integrable and one obtains
\begin{eqnarray*}
J'_\lambda(E_n)(\vp)&=& (E_n,\vp) +\lambda
\int_{\Om} \langle E_n,\vp\rangle\,dx
    - \int_{\Om}\langle |E_n|^{p-2}E_n,\vp\rangle\,dx\\
    &\to& J'_\lambda(E_0)(\vp),
\end{eqnarray*}
which completes the proof of the weak-to-weak$^*$ continuity.
Moreover if $E_0\in\cN_\lambda^{cyl}$, then
\begin{eqnarray*}
\liminf_{n\to\infty} J_\lambda(E_n)&=& \liminf_{n\to\infty}\Big(\frac12-\frac1p\Big)
    \int_{\Om} |E_n|^{p}\,dx\\
    &\geq& \Big(\frac12-\frac1p\Big)
    \int_{\Om} |E_0|^{p}\,dx\\
    &=& J_\lambda(E_0),
\end{eqnarray*}
so that (C3) holds.
\end{proof}

\subsection{Compactly perturbed problem and proof of Theorem \ref{thm:main}}

We take $X^0 :=  \tcV$, $X^1 := \cW$ and let us
consider the functional $J_{cp}:X=\cV\oplus\cW\to\R$ given by
\begin{eqnarray}\label{eq:defOfJ0}
J_{cp}(E)&=&\frac12\int_\Om |\curlop E|^2\, dx + \frac{\lambda}{2}\int_{\Om} |w|^2\, dx-\frac{1}{p}\int_\Om |E|^p\, dx\\\nonumber
&=&\frac12\|v\|^2-I_{cp}(E),\hbox{ for }E=v+w\in\cV\oplus\W,
\end{eqnarray}
where $I_{cp}(E)=-\frac{\lambda}{2}\int_{\Om} |w|^2\, dx+\frac{1}{p}\int_\Om |E|^p\, dx$. Moreover we define the corresponding Nehari-Pankov manifold
$$\cN_{cp}=\{E\in (\cV\oplus\W)\setminus \W:\; J_{cp}'(E)|_{\R E\oplus\W}=0\}.$$
Observe that as in Lemma \ref{eq:LemA1A4} and in Lemma \ref{eq:LemB2C2} we show that $J_{cp}$ satisfies the corresponding conditions (A1)-(A4) and (B1)-(B3). Moreover if $E\in \V\oplus \W$, $w\in \W$ and $t\geq 0$ then
\begin{equation}\label{eq:LemB3checkeqCP}
J_{cp}(E)\geq J_{cp}(tE+w) -J'_{cp}(E)\left(\frac{t^2-1}{2}E+tw\right).
\end{equation}
Since for $E=v+w$ we have
$$J_\lambda(E)-J_{cp}(E)=\frac{\lambda}{2}\int_{\Om}|v|^2\,dx,$$
and $\V$ is compactly embedded into $L^2(\Om,\R^3)$, then we easily show that condition (C1) holds.

\begin{Lem}\label{eq:lemCoercive}
$J_\lambda$ is coercive on $\cN_\lambda$ and $J_{cp}$ is coercive on $\cN_{cp}$.
\end{Lem}

\begin{proof}
Let $E_n=v_n+w_n\in \cN_\lambda$ and suppose that $\|E_n\|\to\infty$. Observe that
$$J_\lambda(E_n)=J_\lambda(E_n)-\frac{1}{2}J'_\lambda(E_n)(E_n)=\Big(\frac12-\frac1p\Big)
|E_n|^p_p\geq C_1 |w_n|_p^p$$
for some constant $C_1>0$, since $\W$ is closed, $\cl\V\cap\W=\{0\}$ in $L^p(\Om,\R^3)$ and the projection $\cl\V\oplus\W$ onto $\W$ is continuous. Hence, if $|E_n|_p\to\infty$, then $J_\lambda(E_n)\to\infty$ as $n\to\infty$. Suppose that $|E_n|_p$ is bounded. Then $\|v_n\|\to\infty$ and 
\begin{eqnarray*}
J_\lambda(E_n)=J_\lambda(E_n)-\frac{1}{p}J'_\lambda(E_n)(E_n)&=&\Big(\frac12-\frac1p\Big)
\Big(\int_{\Om}|\curlop v_n|^2\,dx + \lambda\int_{\Om} |v_n+w_n|^2\,dx
\Big)\\
&\geq&\Big(\frac12-\frac1p\Big)
\Big(\int_{\Om}|\curlop v_n|^2\,dx + \lambda C_2|E_n|^2_p
\Big),
\end{eqnarray*}
for some constant $C_2>0$.
Thus $J_\lambda(E_n)\to\infty$. Similarly we show that $J_{cp}$ is coercive on $\cN_{cp}$.
\end{proof}

\begin{Lem}\label{LemEstimateLevels}
The following inequalities hold
\begin{eqnarray*}
c_\lambda&=&\inf_{\cN_\lambda}J_\lambda \leq \Big(\frac12-\frac1p\Big)(\lambda+\la_{\nu})^{\frac{p}{p-2}}\mu(\Om),\\
c_0&\geq& d_\lambda:=\inf_{\cN_{cp}}J_{cp}\geq \Big(\frac12-\frac1p\Big)S^{\frac{p}{p-2}}.
\end{eqnarray*}
\end{Lem}
\begin{proof}
Let $e_\nu\in \V^+$ be an eigenvector corresponding to $\lambda_\nu$.
Then $te_\nu+\tv+w\in\cN_\lambda$ for some $t>0$, $\tv\in\tcV$ and $w\in\W$. Let $v=te_\nu+\tv$ and observe that
\begin{eqnarray*}
c_\lambda&\leq&J_\lambda(te_\nu+\tv+w)\\
&=&
\frac{\la_{\nu}}{2}\int_{\Om}|te_\nu|^2\, dx +\frac{1}{2}\int_{\Om}|\curlop \tv|^2\,dx + \frac\la2\int_{\Om}|v+w|^2\,dx
    - \frac1p\int_{\Om}|v+ w|^p\,dx\\
&\leq& 
\frac{\la_{\nu}}{2}\int_{\Om}|v|^2\, dx + \frac\la2\int_{\Om}|v+w|^2\,dx
    - \frac1p\int_{\Om}|v+w|^p\,dx\\
&\leq& 
\frac{\lambda+\la_{\nu}}{2}\int_{\Om}|v+w|^2\,dx
    - \frac{1}{p}\int_{\Om}|v+w|^p\,dx\\
    &\leq& 
\frac{\lambda+\la_{\nu}}{2}\mu(\Om)^{\frac{p-2}{p}}\Big(\int_{\Om}|v+w|^p\, dx\Big)^{\frac{2}{p}}
    - \frac{1}{p}\int_{\Om}|v+w|^p\,dx\\
&\leq& \Big(\frac12-\frac1p\Big)(\lambda+\la_{\nu})^{\frac{p}{p-2}}\mu(\Om),
\end{eqnarray*}
where the last inequality follows from the following inequality $\frac{A}{2}t^2-\frac{1}{p}t^p\leq \big(\frac12-\frac1p\big)A^{\frac{p}{p-2}}$ for $t\geq 0$ and $A>0$.
Now let
 $E=v+w\in\cN_{cp}$. Note that by \eqref{eq:LemB3checkeqCP}
we show the first inequality
$$J_{cp}(E)\geq J_{cp}(tv)
\geq \frac12 t^2 S\Big(\int_{\Om}|v|^p\, dx\Big)^{\frac{2}{p}}-\frac1p t^p \int_{\Om}|v|^p\, dx\geq \Big(\frac12-\frac1p\Big)S^{\frac{p}{p-2}},$$
and the last one is obtained by taking
$$t:=S^{\frac{1}{p-2}}\Big(\int_{\Om}|v|^p\, dx\Big)^{-\frac{1}{p}}>0.$$
Thus $d_\lambda\geq \Big(\frac12-\frac1p\Big)S^{\frac{p}{p-2}}$. Moreover if $E\in\cN_0$ then there are $t>0$ and $w\in \W$ such that $tE+w\in\cN_{cp}$ and $J_0(E)\geq J_0(tE+w)\geq J_{cp}(tE+w)\geq d_\lambda$. Therefore $c_0\geq d_\lambda$.
\end{proof}

\begin{altproof}{Theorem \ref{thm:main}}
Note that if $\lambda+\lambda_\nu<S\mu(\Om)^{\frac{2-p}{p}}$, then by Lemma \ref{LemEstimateLevels} we get $c_\lambda<d_\lambda\leq c_0$ and statement a) follows from Theorem \ref{ThLink2} a).  Since $J_\lambda$ satisfies \eqref{eq:thastract2eq1}, by Theorem \ref{ThContinuityuOfClambda},  the function $(-\lambda_{\nu},-\lambda_{\nu-1}]\ni\lambda\mapsto c_\lambda\in (0,+\infty)$ is non-decreasing and continuous. Suppose that $-\lambda_{\nu}<\mu_1<\mu_2\leq-\lambda_{\nu-1}$, $c_{\mu_1}=c_{\mu_2}$ and $c_\lambda$ is attained for some $\lambda\in (\mu_1,\mu_2]$. Then, similarly as in \eqref{eq:Abstractest1}, we show that $c_\lambda>c_{\mu_1}$, which is a contradiction. In view of Lemma \ref{LemEstimateLevels}, we infer that $c_\lambda\to 0$ as $\lambda\to\lambda_\nu^-$.
\end{altproof}

\subsection{Compactly perturbed problem and proof of Theorem \ref{thm:main2}}

In the cylindrically symmetric case, we take $X^0 :=  \tcV^{cyl}$, $X^1 := \{0\}$ and
we consider the  compactly perturbed functional $J_{cp}|_{X^{cyl}}=J_0|_{X^{cyl}}:X^{cyl}\to\R$, so that $J_{cp}|_{X^{cyl}}$ is independent on $\lambda$ and
$$
J_{cp}(E)=J_{0}(E)
=\frac12\|E\|^2-I_{cp}(E),\hbox{ for }E\in X^{cyl}\subset \cV,
$$
where $I_{cp}(E)=I_0(E)=\frac{1}{p}\int_\Om |E|^p\, dx$. We consider the corresponding Nehari-Pankov manifold
$$\cN_{cp}:=\cN_0^{cyl}=\{E\in X^{cyl}\setminus \{0\}:\; J_{cp}'(E)|_{\R E}=0\},$$
which coincides with the usual Nehari manifold of $J_0|_{X^{cyl}}$.
Similarly as in Lemma \ref{eq:LemA1A4} and in Lemma \ref{eq:LemB2C2} we check that $J_{cp}|_{X^{cyl}}$ satisfies the corresponding conditions (A1)-(A4) and (B1)-(B3).\\ 
\indent Since for $E\in X^{cyl}\subset\V$ we have
$$J_\lambda(E)-J_{cp}(E)=\frac{\lambda}{2}\int_{\Om}|E|^2\,dx,$$
and $\V$ is compactly embedded into $L^2(\Om,\R^3)$, then we easily show that condition (C1) holds. Observe that $J_\lambda$ is also coercive on $\cN_\lambda^{cyl}$ and the similar estimates as in Lemma \ref{LemEstimateLevels} are satisfied.

\begin{Lem}\label{LemEstimateLevelsCyl}
The following inequalities hold
\begin{eqnarray*}
 c_\lambda^{cyl}&=&\inf_{\cN_\lambda^{cyl}}J_\lambda\leq \Big(\frac12-\frac1p\Big)(\lambda+\la_{\nu}^{cyl})^{\frac{p}{p-2}}\mu(\Om)\\
c_0^{cyl}&:=&\inf_{\cN_{0}^{cyl}}J_{0}\geq \Big(\frac12-\frac1p\Big)S^{\frac{p}{p-2}}.
\end{eqnarray*}
\end{Lem}

We show that the Palais-Smale condition in $\cN_\lambda^{cyl}$ holds for sufficiently small $\beta$ in the norm topology, hence in the corresponding $\cT$-topology .

\begin{Lem}\label{LemPScond}
$J_\lambda|_{X^{cyl}}$ satisfies the $(PS)_\beta^{\cT}$-condition in $\cN_\lambda^{cyl}$ for 
$$\beta<\Big(\frac12-\frac1p\Big)S^{\frac{p}{p-2}}.$$
\end{Lem}
\begin{proof}
Let $(E_n)$ be a $(PS)_\beta$-sequence such that $(E_n)\subset \cN_\lambda^{cyl}$. By the coercivity of $J_\lambda$ we observe that $(E_n)$ is bounded and we may assume that $E_n\weakto E_0$ in $X^{cyl}$. In view of (V) and since $X_N(\Om)$ is compactly embedded into $L^2(\Om,\R^3) $ we have that $E_n\to E_0$ in $L^{2}(\Om,\R^3)$ and $E_n(x)\to E_0(x)$ for a.e. $x\in\Om$ passing to a subsequence if necessary. Hence by the weak-to-weak$^*$ continuity of $J'_\lambda$ on $\cN_\lambda^{cyl}$ we infer that $J'_\lambda|_{X^{cyl}}(E_0)=0$. In view of the Brezis-Lieb lemma \cite{BrezisLieb}  we get
$$
\lim_{n\to\infty}\int_{\Om} |E_n|^{p} \,dx -
\int_{\Om} |E_n-E_0|^{p} \,dx
=\int_{\Om} |E_0|^{p} \,dx,$$
hence
\begin{equation}\label{eq:PS1}
\lim_{n\to\infty}\big(J_\lambda(E_n)-J_\lambda(E_n-E_0)\big)=J_\lambda(E_0)\geq 0
\end{equation}
and
\begin{equation}\label{eq:PS22}
\lim_{n\to\infty}\big(J_{0}'(E_n)(E_n)-J_{0}'(E_n-E_0)(E_n-E_0)\big)=J_{0}'(E_0)(E_0).
\end{equation}
Since $E_n\in\cN_\lambda^{cyl}$ and $E_n\to E_0$ in $L^2(\Om,\R^3)$ then
$$J_{0}'(E_n)(E_n)=J'_\lambda(E_n)(E_n)-\lambda\int_{\Om}|E_n|^2\,dx\to 
-\lambda\int_{\Om}|E_0|^2$$
as $n\to\infty$
and 
$$J_{0}'(E_0)(E_0)=J'_\lambda(E_0)(E_0)-\lambda\int_{\Om}|E_0|^2\,dx=
-\lambda\int_{\Om}|E_0|^2\,dx.$$
Therefore by \eqref{eq:PS22}
\begin{equation}\label{eq:PS2}
J_{0}'(E_n-E_0)(E_n-E_0)\to 0
\end{equation}
as $n\to\infty$. 
Note that in view of \eqref{eq:LemB3checkeqCP} we get the following inequality
\begin{equation}\label{eq:PS3}
J_{0}(E_n-E_0)\geq J_{0}(t(E_n-E_0))+J_{0}'(E_n-E_0)\Big(\frac{t^2-1}{2}(E_n-E_0)\Big)
\end{equation}
for any $t\geq 0$.
Suppose that $\liminf_{n\to\infty}\|E_n-E_0\|>0$. Then passing to a 
subsequence 
$$\lim_{n\to\infty}\|E_n-E_0\|>0\hbox{ and }\inf_{n\geq 1}\|E_n-E_0\|>0.$$
Note that, since $J'_\lambda(E_n)(E_n-E_0)\to 0$ we get 
$\liminf_{n\to\infty}|E_n-E_0|_p>0$ and we may assume that 
$$\inf_{n\geq 1}|E_n-E_0|_p>0.$$
Hence 
$$t_n:=\Big(\frac{\|E_n-E_0\|^2}{\int_{\Om}|E_n-E_0|^p\,dx}\Big)^{\frac{1}{p-2}}>0$$
is bounded and
$$J_{0}(t_n(E_n-E_0))\geq\Big(\frac12-\frac1p\Big)S^{\frac{p}{p-2}}.$$
Then \eqref{eq:PS1}, \eqref{eq:PS2} and \eqref{eq:PS3} imply that
$$\beta\geq \lim_{n\to\infty}J_\lambda(E_n-E_0)=
\lim_{n\to\infty}J_{0}(E_n-E_0)\geq \lim_{n\to\infty}J_{0}(t_n(E_n-E_0))\geq\Big(\frac12-\frac1p\Big)S^{\frac{p}{p-2}},$$
which is a contradiction. Therefore, passing to a subsequence, $E_n\to E_0$, hence also in the $\cT$- topology.
\end{proof}

\begin{altproof}{Theorem \ref{thm:main2}}
Let $E\in X^{cyl}$ be an eigenvector corresponding to $\lambda_\nu^{cyl}$ such that $|E|_p=1$. Then, by the definition of $S$ 
 and by the H\"older inequality, we get 
 $$S\mu(\Om)^{\frac{2-p}{p}}\leq \mu(\Om)^{\frac{2-p}{p}}\int_{\Om}|\curlop E|^2\,dx 
 =\lambda_\nu^{cyl}\mu(\Om)^{\frac{2-p}{p}}\int_{\Om}|E|^2\,dx\leq \lambda_\nu^{cyl}.$$
Then, in view of Lemma \ref{LemEstimateLevelsCyl}, we may define
$$\eps_\nu:=\sup\{\eps\in (0,\lambda_\nu^{cyl}]: c_{\lambda_\nu+\eps}^{cyl}<c_0^{cyl}\}\geq S\mu(\Om)^{\frac{p}{p-2}}.$$
By Theorem \ref{ThContinuityuOfClambda}, we infer that $(-\lambda_\nu^{cyl},-\lambda_{\nu-1}^{cyl}]\ni \lambda\mapsto c_{\lambda}^{cyl}\in (0,+\infty)$ is continuous and non-decreasing. Hence $c_{\lambda}^{cyl}<c_0^{cyl}$ for $\lambda\in(-\lambda_{\nu}^{cyl},-\lambda_{\nu}^{cyl}+\eps_\nu)$ and by Theorem \ref{ThLink2} b) we obtain that $c_\lambda^{cyl}$ is attained by a critical point of $J_\lambda$, thus a) is proved. Again, by Theorem \ref{ThContinuityuOfClambda} we show that  the function $(-\lambda_{\nu}^{cyl},-\lambda_{\nu}^{cyl}+\eps_\nu]\cap (-\lambda_\nu^{cyl},-\lambda_{\nu-1}^{cyl}]\ni\lambda\mapsto c_\lambda^{cyl}\in (0,+\infty)$ is continuous and strictly increasing. Hence, taking into account also Lemma \ref{LemEstimateLevelsCyl} we get statement c). If $\eps_\nu<\lambda_\nu^{cyl}-\lambda_{\nu-1}^{cyl}$, then $c_\lambda^{cyl}$ is not attained for $\lambda\in(-\lambda_{\nu}^{cyl}+\eps_\nu,-\lambda_{\nu-1}^{cyl}]$. Indeed, if $c_\lambda^{cyl}$ is attained  and  $\lambda> -\lambda_{\nu}^{cyl}+\eps_n$, then, arguing as in \eqref{eq:Abstractest1}, we get
$$c_0^{cyl}\geq c_\lambda^{cyl}>c_\mu^{cyl}$$
for some $\mu\in (-\lambda_{\nu}^{cyl}+\eps_\nu,\lambda)$, which contradicts the definition of $\eps_\nu$. Hence $c_\lambda^{cyl}$ is not attained and the function  for $\lambda\mapsto c_\lambda^{cyl}\in (0,+\infty)$  is constant for $\lambda\in[-\lambda_{\nu}^{cyl}+\eps_\nu,-\lambda_{\nu-1}^{cyl}]$. The proof of b)  is complete. Now we show d). Let
$$A(\lambda):=\big\{k\geq 1: -\lambda^{cyl}_k <\lambda <-\lambda^{cyl}_k + S\mu(\Om)^{\frac{2-p}{p}}\hbox{ and }\lambda_k>\lambda_{k-1}\}$$
and observe that $\tilde{m}(\lambda)=\sum_{k\in A(\lambda)}m(\lambda_k^{cyl})$.
In view of Lemma \ref{LemPScond}, $J_\lambda|_{X^{cyl}}$ satisfies the $(PS)_\beta^{\cT}$-condition in $\cN_\lambda^{cyl}$ for $$\beta<\beta_0:=\Big(\frac12-\frac1p\Big)S^{\frac{p}{p-2}}.$$
We estimate $m(\cN_\lambda^{cyl},\beta_0)$ from below. 
Let $\V(\lambda_k^{cyl})$ denote the eigenspace corresponding to $\lambda_k^{cyl}$ such that  $\dim\V(\lambda_k^{cyl})=m(\lambda_k^{cyl})$, $k\in A(\lambda)$.
Let 
$$S(\lambda):=S\big(\bigoplus_{k\in A(\lambda)}\V(\lambda_k^{cyl})\big)$$
be the unit sphere in $\bigoplus_{k\in A(\lambda)}\V(\lambda_k^{cyl})\subset \V^{cyl+}$.  Then we may define the continuous map $h:S(\lambda)\to \cN_\lambda^{cyl}$ such  that
$$h(E)=n_{\lambda}(E)\quad\hbox{for }E\in S(\lambda),$$
where $n_\lambda:S(\lambda)\to \cN_\lambda^{cyl}$ is the homeomorphism 
given after Proposition~\ref{prop:nehari}. 
Since  $J_\lambda$ is even,  $h$ is odd. Similarly as in Lemma \ref{LemEstimateLevelsCyl}, we show that for $E\in S(\lambda)$
$$J_\lambda(h(E))\leq  \max_{k\in A(\lambda)}\Big(\frac12-\frac1p\Big)(\lambda+\la^{cyl}_{k})^{\frac{p}{p-2}}\mu(\Om)=:\beta$$
and thus 
$$h(S(\lambda))\subset J_\lambda|_{X^{cyl}}^{-1}((0,\beta])\cap \cN_\lambda^{cyl}.$$
Observe that $J_\lambda|_{X^{cyl}}^{-1}((0,\beta])\cap \cN_\lambda^{cyl}$ is closed, symmetric and
$$\gamma(J_\lambda|_{X^{cyl}}^{-1}((0,\beta])\cap \cN_\lambda^{cyl})\geq 
\gamma(S(\lambda))=\tilde{m}(\lambda).$$
Since $\beta<\beta_0$, we obtain
$$m(\cN_\lambda^{cyl},\beta_0)\geq \tilde{m}(\lambda).$$
Therefore, in view of Theorem \ref{ThLink2} c),
there are at least $\tilde{m}(\lambda)$ pairs of critical points $E$ and $-E$ in $\cN_\lambda^{cyl}$. Now let 
 $\mu_n\to\mu_0$ in $I$ as $n\to\infty$, and take any sequence of symmetric ground states $E_n$ of $J_{\mu_n}$ for $n\geq 1$. Let $\mu=\inf\{\mu_n:n\geq 1\}$ and note that 
 $$J_\mu(E_n)\leq J_{\mu_n}(E_n)=c_{\mu_n}\to c_{\mu_0}$$
 as $n\to\infty$.
By the coercivity  of $J_\mu$ we obtain that $(E_n)$ is bounded, and
observe that
$$J_{\mu_0}(E_n)=c_{\mu_n}+\frac{\mu_0-\mu_n}{2}\int_{\Om}|E_n|^2\,dx\to c_{\mu_0}.$$
Moreover 
$$J_{\mu_0}'(E_n)(E)=(\mu_0-\mu_n)\int_{\Om}\langle E_n, E\rangle\,dx
\leq (\mu_0-\mu_n) C \|E\|$$
for $E\in X^{cyl}$, where $C>0$ is a constant  independent on $n$.
Thus  $J_{\mu_0}'(E_n)\to 0$ and $(E_n)$ is a Palais-Smale sequence at level $c_{\mu_0}$ of $J_{\mu_0}$. Since $\mu_0<\beta_0$, arguing as in Lemma \ref{LemPScond}, we infer that passing to a subsequence $E_n\to E_0$ in $X^{cyl}$. Then we easily see that $E_0$ is a symmetric ground state of $J_{\mu_0}$.
\end{altproof}


\section{Anisotropic and cylindrically symmetric media}\label{sec:AM}

\indent In the last section we present results on more general uniaxial (anisotropic and $G$-invariant) domains  \cite{BartschMederski2,FundPhotonics,StuartZhou03}, which 
are obtained,  again by means of the methods of Section \ref{sec:Nehari}. Namely, we look for solutions of the following equation
\begin{equation}\label{eq:aniso}
\curlop\left(\mu(x)^{-1}\curlop E\right) - V(x)E = f(x,E)\quad\hbox{in }\Om
\end{equation}
together with boundary condition \eqref{eq:bc}. Recall that the anisotropic Brezis–Nirenberg-type variant of \eqref{eq:BNproblem} has been recently studied by Clapp, Pistoia and Szulkin \cite{ClappSzulkinPistoia}. In \eqref{eq:aniso},
the permeability tensor $\mu(x)$ is of the form
\begin{equation}\label{eq:formOfeps1}
\mu(x) = \begin{pmatrix}a(x) & 0 & 0\\0 & a(x) & 0\\0 & 0 & b(x)\end{pmatrix},
\end{equation}
with $a,b\in L^{\infty}(\Om)$ positive, bounded away from $0$, and invariant with respect to the action of $G$ on $\Om$; similarly for $V(x)=\omega^2\eps(x)$, hence for the permittivity tensor $\eps(x)$.
For simplicity, we restrict our considerations to the following nonlinearity $f(x,E)=\partial_E F(x,E)$, where $$F(x,E)=|\Gamma(x) E|^p,$$ 
$2<p\leq 6$ and $\Gamma(x)$ has the similar matrix form \eqref{eq:formOfeps1}, i.e.  the corresponding coefficients $a_\Gamma,b_\Gamma\in L^{\infty}(\Om)$ are positive, bounded away from $0$, and invariant with respect to the action of $G$. The interested reader may play with other types of nonlinearities similarly as in \cite{BartschMederski2}, such that the Nehari-Pankov manifold can be set up for \eqref{eq:aniso}. As opposed to \cite{BartschMederski2}, we do not assume that the following subspace
\[
\cV = \left\{v\in H_0(\curl;\Om): \int_\Om\langle V(x)v,\varphi\rangle\,dx=0
\text{ for every $\varphi\in \cC^\infty_0(\Om,\R^3)$ with $\curlop\varphi=0$} \right\}.
\]
is compactly embedded into $L^p(\Om,\R^3)$ for some $2<p<6$, but we work with condition $(V)$, in which we use the above definition of $\V$ for the anisotropic media.
Hence the energy functional $J:W_0^p(\curl;\Om)\to\R$ given by
$$J(E)=\frac12\int_\Om \langle \mu(x)^{-1}\curlop E, E\rangle\, dx + \frac{1}{2}\int_{\Om} \langle V(x)E,E\rangle\, dx-\int_\Om |\Gamma(x)E|^p\, dx$$
is of $\cC^1$-class and $G$-invariant.
Similarly as above, we may define a subspace $X^{cyl}\subset \V$ of fields of the form \eqref{eq:main} (see \cite{BartschMederski2}[Lemma 6.2]), and there is a discrete sequence of anisotropic eigenvalues 
$$0<\la_1^{cyl}\le \la_2^{cyl}\le\dots\le\la_k^{cyl}\to\infty$$ of the following problem
$$
\curlop(\mu(x)^{-1}\curlop v) = \la V(x)v
$$
in $X^{cyl}$ with the corresponding finite multiplicities $m(\lambda_k^{cyl})\in\N$; see \cite{BartschMederski2}[Corollary 3.3]. We find two closed subspaces $\cV^{cyl+}$ and $\tcV^{cyl}$ of $X^{cyl}$ such that the quadratic form
$$
Q(v):=\int_\Om \left(\langle\mu(x)^{-1}\curlop v,\curlop v\rangle -
\langle V(x)v,v\rangle\right)\,dx,
$$
is positive on $\cV^{cyl+}$ and semi-negative on $\tcV^{cyl}$. Then $\tcV^{cyl}$ is the finite sum of the eigenspaces associated to all $\la_k^{cyl}\le1$, and $\cV^{cyl+}$ is the infinite sum of the eigenspaces associated to the eigenvalues $\la_k^{cyl}>1$. Here $\tcV^{cyl}=\{0\}$ if $\la_1^{cyl}>1$. Then, the symmetric Nehari-Pankov manifold in the anisotropic case is given by 
\begin{equation}\label{eq:DefOfNcylAni}
\cN:=\{E\in X^{cyl}\setminus\tcV^{cyl}:\; J'(E)|_{\R E\oplus\tcV^{cyl}}=0\},
\end{equation}
and
$$c:=\inf_{\cN}J.$$
Since $\mu$ is of the form \eqref{eq:formOfeps1}, we define 
\begin{eqnarray*}
\mu_{\infty}&:=&\max\{|a|_\infty,|b|_\infty\}
\end{eqnarray*}
and similarly we get $V_\infty$ and $\Gamma_\infty$ for $V$ and $\Gamma$ respectively. Moreover let 
$$\Gamma_0:=\inf_{x\in\Om}\min\{|a_\Gamma(x)|,|b_\Gamma(x)|\}.$$
The existence result reads as follows.
\begin{Th}\label{thm:main3}
Suppose that 
\begin{equation}\label{eq:thmain3}
\tilde{m}:=\sharp\big\{k: 0<(\lambda_k^{cyl}-1)V_\infty\mu_\infty\big(\Gamma_\infty/\Gamma_0\big)^{\frac{2}{p}} < S\mu(\Om)^{\frac{2-p}{p}}\big\}\geq 1.
\end{equation}
Then $c>0$ and there is a symmetric ground state solution to \eqref{eq:aniso}, i.e. $c$ is attained by a critical point of $J$. Moreover,
there are at least $\tilde{m}$ pairs of solutions $E$ and $-E$ to \eqref{eq:aniso} of the form \eqref{eq:sym}. 
\end{Th}
\begin{proof}
Note that $X^{cyl}$ is a Hilbert space with the scalar product
	\begin{equation*}
	\langle u,v \rangle_\mu := \int_{\Om}\langle \mu(x)^{-1}\curlop u,\curlop v\rangle\, dx
	\end{equation*}
and $\cV^{cyl+}$ and $\tcV^{cyl}$ are orthogonal with respect to $\langle \cdot,\cdot \rangle_\mu$ and with respect to  the scalar product in $L^2(\Om,\R^3)$ given by
\begin{equation*}
\langle u,v \rangle_V := \int_{\Om}\langle V(x) u,v\rangle\, dx.
\end{equation*}
We consider  the norm $\|v\|_\mu = 	\langle v,v \rangle_\mu$ in $\V$, which is equivalent with the usual one.
Observe that for $E=v^++\tv\in X^{cyl}=\V^{cyl+}\oplus\tcV^{cyl}$ we get
$$J(E)=\frac{1}{2}\|v^+\|_\mu^2 - I(E),$$
where 
$$I(E):=-\frac12\|\tv\|_\mu^2 + \frac12\int_\Om\langle V(x)E,E\rangle\,dx
 + \int_\Om |\Gamma(x)E|^p\,dx.$$
Therefore $J:X^{cyl}\to \R$ is of the form \eqref{EqJ}, and similarly as in Section \ref{sec:varsetting}, we check that assumptions (A1)-(A4), (B1)-(B3), (C1)-(C3) are satisfied, where the compactly perturbed functional is given as follows
$$J_{cp}(E)=\frac{1}{2}\|E\|_\mu^2 - I_{cp}(E),$$
where 
$$I_{cp}(E):= \int_\Om |\Gamma(x)E|^p\,dx$$
for $E\in X^{cyl}$.
Similarly as in Lemma \ref{LemEstimateLevelsCyl} we check that
$$d:=\inf_{\cN_{cp}}J_{cp}\geq \Big(\frac12-\frac1p\Big)S^{\frac{p}{p-2}}\mu_{\infty}^{-\frac{p}{p-2}}\big(p\Gamma_\infty\big)^{-\frac{2}{p-2}}.$$
On the  other hand, we show that
$$c<\Big(\frac12-\frac1p\Big)\big((\lambda_k^{cyl}-1)|V|_\infty\big)^{\frac{p}{p-2}}\mu(\Om)\big(p\Gamma_0\big)^{-\frac{2}{p-2}}$$
for every $k\geq 1$ such that
\begin{equation}\label{eq:anis_k}
0<(\lambda_k^{cyl}-1)V_\infty\mu_\infty\big(\Gamma_\infty/\Gamma_0\big)^{\frac{2}{p}} < S\mu(\Om)^{\frac{2-p}{p}}.
\end{equation}
Indeed, let  $e_k\in \V^{cyl+}$ be an eigenvector corresponding to $\lambda_k^{cyl}$, i.e. 
	$$Q(e_k)=(\lambda_k^{cyl}-1)\int_{\Om}\langle V(x) e_k,e_k\rangle\, dx$$
for $k\geq 1$ such that \eqref{eq:anis_k} holds.
	Then $v=te_k+\tv\in\cN$ for some $t>0$ and $\tv\in\tcV^{cyl}$.  Observe that
	\begin{eqnarray*}
		c&\leq&J(te_k+\tv)=
	\frac12Q(te_k+\tv)
		- \int_{\Om}|\Gamma(x)v|^p\,dx\\
		&=&	\frac12Q(te_k)+\frac12Q(\tv)
		- \int_{\Om}|\Gamma(x)v|^p\,dx\\
		&\leq& 
		\frac12(\lambda_k^{cyl}-1)\int_{\Om}\langle V(x) e_k,e_k\rangle\, dx
		- \int_{\Om}|\Gamma(x)v|^p\,dx\\
		&\leq& 
		\frac12(\lambda_k^{cyl}-1)\int_{\Om}\langle V(x) v,v\rangle\, dx
		- \int_{\Om}|\Gamma(x)v|^p\,dx\\
		&\leq& 
		\frac12(\lambda_k^{cyl}-1)|V|_{\infty}\int_{\Om}|v|^2\, dx
		- \Gamma_0\int_{\Om}|v|^p\,dx\\
		&\leq& 
			\frac12(\lambda_k^{cyl}-1)|V|_{\infty}\mu(\Om)^{\frac{p-2}{p}}\Big(\int_{\Om}|v|^p\, dx\Big)^{\frac{2}{p}}
			- \Gamma_0\int_{\Om}|v|^p\,dx\\
&\leq& \Big(\frac12-\frac1p\Big)\big((\lambda_k^{cyl}-1)|V|_\infty\big)^{\frac{p}{p-2}}\mu(\Om)\big(p\Gamma_0\big)^{-\frac{2}{p-2}}.
	\end{eqnarray*}
Hence \eqref{eq:thmain3} implies that $c<d$ and in view of Theorem \ref{ThLink2} b) we infer that $c$ is attained by a critical point. Similarly as in Lemma \ref{LemPScond} we show that $J$ satisfies the $(PS)_\beta^{\cT}$-condition in $\cN$ for
$$\beta <\beta_0:=\Big(\frac12-\frac1p\Big)S^{\frac{p}{p-2}}\mu_{\infty}^{-\frac{p}{p-2}}\big(p\Gamma_\infty\big)^{-\frac{2}{p-2}}.$$
 We only check the following variant of the Brezis-Lieb lemma 
\begin{equation}\label{eq:BLlemmaAni}
\lim_{n\to\infty}\int_{\Om} |\Gamma(x)E_n|^{p} \,dx -
\int_{\Om} |\Gamma(x)(E_n-E_0)|^{p} \,dx
=\int_{\Om} |\Gamma(x)E_0|^{p} \,dx
\end{equation}
for $E_n\weakto E_0$ in $L^p(\Om,\R^3)$ and $E_n\to E_0$ a.e. on $\Om$,
and the remaining arguments are analogous. Observe that in view of the Vitali convergence theorem
\begin{eqnarray*}
	&&\int_\Om |\Gamma(x)E_n|^p - |\Gamma(x)(E_n-E_0)|^p\,dx\\
	&&\hspace{1cm}
	= \int_\Om\int_0^1\frac{d}{dt}|\Gamma(x)(E_n+(t-1)E_0)|^p\,dtdx\\
	&&\hspace{1cm}
	= \int_0^1\int_\Om\langle p |\Gamma(x)(E_n+(t-1)E_0)|^{p-2}\Gamma(x)(E_n+(t-1)E_0),
\Gamma(x)E_0\rangle\,dxdt\\
&&\hspace{1cm}\to 
\int_0^1\int_\Om\langle p |\Gamma(x)tE_0|^{p-2}\Gamma(x)tE_0,
\Gamma(x)E_0\rangle\,dxdt\\
&&\hspace{1cm}=
\int_\Om |\Gamma(x)E_0|^p\, dx,
\end{eqnarray*}
hence we get \eqref{eq:BLlemmaAni}. Now, similarly as in the proof of Theorem \ref{thm:main2}, we show that 
$$m(\cN,\beta_0)\geq \tilde{m}$$
and in view of Theorem \ref{ThLink2} c) we conclude.
\end{proof}

{\bf Acknowledgements.}
The author was partially supported by the National Science Centre, Poland (Grant No. 2014/15/D/ST1/03638) and he would like to thank the referee for many valuable
comments helping to improve the paper.

\end{document}